\documentclass[fleq]{amsart}
\usepackage{amssymb,amsmath,latexsym,amsbsy,color,enumerate,comment,enumitem, tikz}
\usepackage{graphicx}
\numberwithin{equation}{section}

\newcommand{\comm}[1]{}
\newtheorem{theorem}{Theorem}
\newtheorem{definition}[theorem]{Definition}
\newtheorem{lemma}[theorem]{Lemma}
\newtheorem{question}[theorem]{Question}
\newtheorem{remark}[theorem]{Remark}
\newtheorem{proposition}[theorem]{Proposition}
\newtheorem{corollary}[theorem]{Corollary}

\numberwithin{theorem}{section}

\newtheorem{thm}[equation]{Theorem}
\newtheorem{cor}[equation]{Corollary}
\newtheorem{lem}[equation]{Lemma}
\newtheorem{prop}[equation]{Proposition}
\newtheorem{eg}[equation]{Example}
\theoremstyle{remark}
\newtheorem{rem}[equation]{Remark}

\DeclareMathOperator{\SO}{SO}
\DeclareMathOperator{\SL}{SL}
\DeclareMathOperator{\GL}{GL}
\DeclareMathOperator{\SU}{SU}
\DeclareMathOperator{\rank}{rank}
\DeclareMathOperator{\Mat}{Mat}
\DeclareMathOperator{\Nbhd}{Nbhd}
\DeclareMathOperator{\Cone}{Cone}
\DeclareMathOperator{\Star}{Star}

\newcommand{\Lie}[1]{\mathfrak{\lowercase{#1}}}
\renewcommand{\sl}{\Lie{sl}}
\newcommand{\so}{\Lie{so}}

\newcommand{\CC}{{\mathbb{C}}}
\newcommand{\QQ}{{\mathbb{Q}}}
\newcommand{\RR}{{\mathbb{R}}}
\newcommand{\ZZ}{{\mathbb{Z}}}

\newcommand{\Rrank}{\rank_{\RR}}
\newcommand{\ints}{\mathcal{O}}
\newcommand{\iso}{\cong}

 \newcounter{case}
 \newenvironment{case}[1][\unskip]{\refstepcounter{case}
 \medskip \noindent \textbf{Case \thecase.}\em\ #1\ }{\unskip\upshape}
 \renewcommand{\thecase}{\arabic{case}}

\numberwithin{equation}{section}

\begin{document}
 \title{Quasi-isometric embeddings of non-uniform lattices}
\author{David Fisher and Thang Nguyen}\thanks{Authors partially supported by NSF grant DMS-1308291.}
    \address{Indiana University Bloomington IN 47405}
 \email{}
\thanks{}
    \date{\today}
    \keywords{}
    \subjclass[2000]{}
    \begin{abstract}
Let $G$ and $G'$ be simple Lie groups of equal real rank and real rank at least $2$.  Let $\Gamma <G$ and $\Lambda < G'$ be non-uniform lattices.   We prove a theorem that often implies that any quasi-isometric embedding of $\Gamma$ into $\Lambda$ is at bounded distance from a homomorphism.  For example, any quasi-isometric embedding of $SL(n,\mathbb Z)$ into $SL(n, \mathbb Z[i])$ is at bounded distance from a homomorphism.  We also include a discussion of some cases when this result is not true for what turn out to be purely algebraic reasons.
\end{abstract}
\maketitle

\section{Introduction}       

The rigidity theorems of Mostow and Margulis are among the most celebrated results about the intersection of discrete groups and geometry.     With the rise of Gromov's program for the geometric study of discrete groups, coarse analogues of these results were among the most desired results \cite{Gromov}.     There are many possible translations of these theorems to a coarse setting, and so are results and questions in this direction  (see \cite{Farb} for a good survey). We first recall two basic definitions:

\begin{definition}
\label{defn:qi} Let $(X,d_X)$ and $(Y,d_Y)$ be metric spaces. Given
real numbers $L{\geq}1$ and $C{\geq}0$,a map $f:X{\rightarrow}Y$ is
called a {\em $(L,C)$-quasi-isometry} if
\begin{enumerate} \item
$\frac{1}{L}d_X(x_1,x_2)-C{\leq}d_Y(f(x_1),f(x_2)){\leq}L d_X(x_1,x_2)+C$
for all $x_1$ and $x_2$ in $X$, and, \item the $C$ neighborhood of
$f(X)$ is all of $Y$.
\end{enumerate}
\noindent If $f$ satisfies $(1)$ but not $(2)$, then $f$ is called a {\em $(L,C)$-quasi-isometric
embedding}.
\end{definition}
\begin{remark}
 Throughout this paper, all semisimple Lie groups will have no compact factors.
\end{remark}

In recent joint work with Whyte, the first author has extended these rigidity first explored by Mostow and Margulis to the context of quasi-isometric embeddings of higher rank symmetric spaces \cite{FiWh}.  As uniform lattices in simple Lie groups are quasi-isometric to the symmetric space associated to the Lie group, that paper can be read as describing quasi-isometric embeddings of uniform lattices.  In this paper we consider the somewhat harder problem of describing quasi-isometric embeddings of non-uniform lattices.  Already for self-quasi-isometries of non-uniform lattices, a striking new phenomenon arose, first discovered by R. Schwartz \cite{Sch}.  This was extended to irreducible lattices in products of rank one groups by Farb-Schwartz and Schwartz and finally to all higher rank lattices by Eskin \cite{Esk,FaSc, Sch2}.  An alternate approach
to some aspects of Eskin's proof by Drutu is important both to the work of Fisher and Whyte and here \cite{Dru}.

\begin{theorem}
\label{theorem:selfqi}
Given a non-uniform lattice $\Gamma$ in a simple noncompact Lie group $G$ not locally isomorphic to $SL(2,\mathbb R)$, any self-quasi-isometry of $\Gamma$ is at
bounded distance from a homomorphism $\Gamma' \rightarrow \Gamma$ where $\Gamma' < \Gamma$ has finite index.
\end{theorem}

We find another striking new phenomenon by extending this one to certain quasi-isometric embeddings of non-uniform lattices. This builds on work of Drutu, Eskin and Fisher-Whyte described above \cite{Dru, Esk, FiWh}.  Given a simple Lie group $G$ of higher real rank, the Cartan subgroup $A$ of $G$ comes with a set of distinguished hyperplanes called {\em Weyl chamber walls}. We refer to the pattern of these walls as the {\em Weyl chamber pattern}. We inherit from \cite{FiWh} an assumption on embeddings of Weyl chamber patterns and prove the following:

\begin{theorem}\label{mainthm}
Let $\Gamma, \Lambda$ be nonuniform lattices in higher rank simple Lie groups $G, G'$ of the same rank and rank at least $2$. Assume:
 \begin{enumerate}
 \item any linear embedding of the Weyl chamber pattern for $G$ into the Weyl chamber pattern for $G'$ is conformal, and
\item there is no closed subgroup $G < H <  G'$ with compact $H$ orbit on $\Lambda \backslash G'$.
 \end{enumerate}
 Then if $\varphi: \Gamma\to\Lambda$ is a QI-embedding, then $\varphi$ is at bounded distance from a homomorphism $\Gamma' \to\Lambda$ where $\Gamma'<\Gamma$ has finite index.
\end{theorem}

We remark that the assumption on orbit closures is necessary. In the absence of this condition the proof of Theorem \ref{mainthm} shows that any quasi-isometric embedding is given by the following simple construction. If there is a closed $H$ orbit in $\Lambda \backslash G'$, this means that (possibly after replacing $H$ with a conjugate) $\Lambda' = H\cap \Lambda$ is cocompact. The inclusion of $\Gamma$ into $G$ is a quasi-isometric embedding by results of Lubotzky-Mozes-Raghunathan and the inclusion of $G<H$ is forced to be an isometric embedding by the ambient assumptions \cite{LMR}.  Since $\Lambda'$ is quasi-isometric to $H$, this gives a quasi-isometric embedding of $\Gamma$ into $\Lambda$.  See below for examples and more discussion.

An immediate consequence of Theorem \ref{mainthm}  is the following strengthening of the main result of Eskin in \cite{Esk}.  This proves that higher rank non-uniform lattices are {\em coarsely co-Hopfian} in the sense introduced by Kapovich and Lukyanenko in \cite{KaLu}.

\begin{corollary}
Let $\Gamma$ be a nonuniform lattice in a simple Lie group $G$ of real rank at least $2$. Then any quasi-isometric embedding $\varphi: \Gamma\to \Gamma$ is an isomorphism on finite index subgroups.
\end{corollary}

\noindent We remark that a careful reading of Eskin's paper reveals that the corollary is already proven there.

In addition we have many results concerning quasi-isometric embeddings of distinct lattices, the simplest of which is:

\begin{corollary}
\label{corollary:sl}
Let $\varphi: SL(n, \mathbb{Z})\to SL(n,\mathbb{Z}[i])$ be a quasi-isometric embedding. Then $\varphi$ is at
bounded distance from a homomorphism $\phi: \Gamma \rightarrow SL(n, \mathbb{Z}[i])$ where $\Gamma < SL(n, \mathbb Z)$ is of finite index.
\end{corollary}

There are a number of other results that follow once one has some idea when given $G, G', \Gamma$ and $\Lambda$ as in the theorem, there is a closed subgroup $H$ containing $\Gamma$ and therefore $G$, such that $H$ has a closed orbit in $\Lambda \backslash G'$.  A partial solution to this problem is given in the appendix to this paper by Garibaldi, McReynolds, Miller and Witte-Morris.  Examples do exist and are constructed in the appendix, and their work also gives some restrictions, yielding results like:

\begin{corollary}\label{orthogonalcorollary}
Let either $m\geq n \geq 2$ or $m + n \geq 7$ and let $\Gamma$ be a non-uniform lattice in $SO(n,m)$ and $\Lambda$ a non-uniform lattice in $SO(n, m+l)$ where $l< n+m$, then any quasi-isometric embedding $\varphi: \Gamma \rightarrow \Lambda$ is at bounded distance from a homomorphism $\phi: \Gamma' \rightarrow \Lambda$ where $\Gamma'<\Gamma$ has finite index.
\end{corollary}

We remark here that while in the context of \cite{FiWh} it is clear that some assumption on Weyl chamber patterns is required when considering quasi-isometric embeddings of symmetric spaces as quasi-isometric embeddings of $SL(n, \mathbb R) \rightarrow SP(2(n-1), \mathbb R)$ and $SL(2, \mathbb R) \times SL(2,\mathbb R) \rightarrow SL(3, \mathbb R)$ are constructed there, but it is less clear that this assumption is needed here.  In particular, we cannot answer:

\begin{question}
Are there any quasi-isometric embeddings of $SL(n, \mathbb Z) \rightarrow SP(2(n-1), \mathbb Z)$?  Are there
any quasi-isometric embeddings of $SL(2, \mathbb Z) \times SL(2, \mathbb Z) \rightarrow SL(3, \mathbb Z)$?
\end{question}

We believe the answer to the first question is no, but a proof requires genuinely new ideas.  Either one would
need to understand all quasi-isometric embeddings of the associate symmetric spaces or one would need to find an approach to the quasi-isometric embeddings of lattices that did not make reference to the symmetric spaces.  Since $SL(2, \mathbb Z)$ is virtually free, the second question seems to admit a much wider array of approaches.

\noindent {\bf Outline of proof and differences from earlier work:}  The main lines of the proof are very similar to those in the papers of Eskin or Drutu, but with some substantial additional difficulties and also some substantial simplifications of arguments \cite{Dru, Esk}. Let $X$ be the symmetric space associated to $G$ and $Y$ the symmetric space associated to $G'$. We begin by showing that the embedding of lattices gives a map sending almost every flat to flat.  In this part of the argument, our argument resembles Drutu's more than Eskin's but simplifies the argument further particularly by using idea from \cite{FiWh}.  We show that almost every flat stays at sublinear distance from the thick part of $G/\Gamma$ and so has a well defined image in the asymptotic cone. Our argument differs from Drutu's in that we do not use the Kleinbock-Margulis logarithm law but use a more naive argument that gives a worse, but still sublinear, bound.  As in Drutu's paper, additional argument is required to show that the set of flats for which this is true is rich enough to capture enough incidences so that we have a full measure family of flats with well-defined maps from the cone of $X$  to the cone of $Y$ which also have chamber walls of any dimension mapping to chamber walls of the same dimension in the image flat. These arguments occur in subsection \ref{subsection:goodflats}.    To show that the image of a flat is a single flat, we use an argument close to the one in the paper by the first author and Whyte and in particular, use the higher rank Mostow-Morse lemma.  This lemma shows that off a set of codimension $2$ the image of any point in the flat has a neighborhood contained in a single flat and greatly simplifies the arguments from \cite{Dru, Esk} see subsection \ref{subsection:flatstoflats}. As in the papers of Drutu and Eskin, the most difficult step is to show that the boundary map we have constructed extends to a continuous morphism of buildings.  In our context there is substantial additional difficulty here, since chambers do not in fact, in general map to chambers and one has instead a map from chambers to finite collections of chambers.  Here we use the fact that the map is isometric along flats and the Tits buiding structure on the boundary of $X$ to show that this yields a well defined map from a set of full measure in the Furstenberg boundary of $X$ to a finite product of Furstenberg boundaries of $Y$.
{ Following Eskin's original argument (also used by Drutu), we use negative curvature to obtain continuity of the map on the set of chambers adjacent to a chamber wall, see section \ref{section:firstcontinuity}. The set of Weyl sectors that are adjacent to a fixed hyperplane can naturally be parametrized as a hyperbolic space, and the embedding coarsely preserves distance hence images of chambers at infinity also vary continuously. This is similar with showing boundary map is continuous in Mostow rigidity for hyperbolic manifolds. As in prior work, because the lattice is non-uniform, we only get this continuity at almost every chamber wall and for almost every chamber adjacent to the chamber wall. A short additional argument is required because our map on chambers is multi-valued. This also makes the next step much harder compared with the quasi-isometry case.

The next step is done in section \ref{section:secondcontinuity}, showing the boundary map extends continuously to a building homomorphism from boundary at infinity of $X$ to a sub-building of boundary at infinity of $Y$. This is the most novel and most difficult part of this paper.  The two buildings $\partial_TX$ and $\partial_TY$ are not isomorphic, so we cannot apply existing result of Tits as Drutu did in \cite[Section 5.3 A and B]{Dru}. In Eskin's approach, an additional problem arises since the Furstenberg boundary of $X$ maps to a very thin set (measure zero, clearly not dense) in the Furstenberg boundary of $Y$ so the arguments of \cite[Section 5.4]{Esk} do not apply.  To overcome this difficulty, we work directly with the building structure at infinity.   Motivated by Tits' \cite[section 4]{Tits74}, we show, by induction on combinatorial distance, that the boundary map extends continuously to an injective adjacent preserving map on balls around a fixed chambers. This is done by first picking a good chamber in the sense that at almost every wall in each sphere (w.r.t.~combinatorial metric) around the chamber we have the continuity obtained in previous step. We also fix a good apartment containing that chamber. The induction argument uses chambers adjacent to two opposite walls. This roughly means we can get an injective continuous map on chambers adjacent to a wall if there is an opposite wall and an injective continuous map defined on adjacent chambers of the opposite wall. Moreover, to make sure that the map constructed by induction argument agrees almost everywhere with our boundary map and has desired properties (injectivity, continuity, and combinatorially well-behaved), in each step of induction argument we also have to show some combinatorial and continuity claims (see proof of Theorem \ref{buildingmap}). As a result we get an extended injective continuous boundary map which also preserves combinatorial structure of $\partial_TX$. In other word, we get a subset of $\partial_TY$ carrying a building structure of the same type of $\partial_TX$, and is homeomorphic to $\partial_TX$ as buildings. After this, we can identify the image of $X$ in $Y$ as a subsymmetric space in $Y$ using the results in \cite{KL2} as in \cite{FiWh}.
}

The rest of the argument resembles that given in \cite{Esk} (and essentially repeated in \cite{Dru}) using Ratner's theorem, but with some additional difficulties, since $G \neq G'$.  It is at this step that the group $H$ arises and the question of compact $G$ invariant sets in $G'/\Lambda$ intervenes.  By Ratner's theorem, these compact invariant sets are homogeneous and the question reduces to finding subgroups $H$ in $G'$ with $G<H$ and $\Lambda \cap H$ a cocompact lattice.  This question is analyzed in the appendix by Garibaldi, McReynolds, Miller and Witte Morris and answers are given in many cases, including those required to prove the specific results stated as Corollary \ref{corollary:sl} and Corallary \ref{orthogonalcorollary}.

\section{Preliminaries}\label{Preliminaries}

In this section, we introduce some notation and terminology.

Let $X,Y$ be the symmetric spaces corresponding to $G,G'$. Let $K,K'$ be maximal compact subgroups in $G$ and $G'$. Let $\pi: G\to G/K=X$, $p:G\to \Gamma\backslash G$, $\overline{p}:G/K \to \Gamma\backslash X$, $\overline{\pi}: \Gamma\backslash G\to \Gamma\backslash X$ are projections.

Let $A$ be the Cartan subgroup of $G$, and let $\Xi$ be the root system associated to $G$. For $\sigma\subset \Xi$, let
\[A_\sigma = \{a\in A| \alpha(\log a)=0, \forall \alpha \in \sigma\}.\]
This is a subflat in the flat $A$. When $\sigma=\varnothing$, $A_\varnothing=A$. When $\sigma=\{\alpha\}$ for any $\alpha\in\Xi$, we also denote $A_\alpha = A_{\{\alpha\}}$. For any $\alpha\in \Xi$, fix a $k_\alpha\in K$ such that two flats $\pi(A)$ and $\pi(k_\alpha A)$ intersect exactly at $\pi(A_\alpha)$. For convenience, we denote $k_\varnothing = 1$. A copy of $\pi(A_\sigma)$ is called the Weyl hyperplane associated to $\sigma\in\Xi$. By Weyl pattern at a point $x$  in a flat $F$, we mean the pattern of Weyl hyperplanes in $F$ passing through $x$. Let $W$ be a chamber, we denote by $W(\infty)$ the boundary at infinity of $W$. This is again a chamber in the building $\partial X$. We use similar notations to denote the boundary at infinity of a flat, a hyperplane, or a ray.

Let $\Xi_+\subset \Xi$ be the set of positive roots and let $A_+=\{a\in A: \alpha(\log a)>0, \forall \alpha\in\\ \Xi_+\}$. Then any chamber in $X$ will have form $\pi(gA_+)$ for some $g\in G$. If $\Delta\subset \Xi$, denote $D_\Delta^+=\{a\in A|\alpha(\log a)>0,\forall \alpha\in\Delta\}$. Denote $U_\Delta$ the unipotent subgroup of $G$ corresponding with the set of roots $\Delta$.

Let $M$ be the subgroup of $K$ consisting all elements that commute with all $\alpha\in \Xi$. Then the Furstenberg boundary of $X$ can be identified with $K/M=:\overline{K}$. Hence there is a natural measure on $\overline{K}$. Also, we denote $\overline{K'}$ for the Furstenberg boundary of $Y$.

For $\alpha\in \Xi$, denote by $P_\alpha$ the parabolic subgroup associated to the root $\alpha$. We have the Langlands decomposition $P_\alpha=M_\alpha A_\alpha N_\alpha$. Let $K_\alpha=K\cap M_\alpha$. $K_\alpha$ is a stabilizer of a fixed face in $\partial X$. There is a natural labeling map that is invariant under the action of Weyl group. And the set of faces in $\partial X$ of the same type as the face $K_\alpha$ can be identify with $K/K_\alpha$. For a face $O=kK_\alpha$, the star chamber of $O$, subset of $\overline{K}$ consists of chambers containing $O$ as a face, can be identified with $kK_\alpha$. This a a copy of a compact group, thus, there is a natural measure on each star chamber.

We use various notions of distance in this paper. Here is the list:
\begin{itemize}
\item $d(.,.)$ stands for distance in $X$ or $Y$.
\item $d_{Hau}(.,.)$ stands for Hausdorff distance between compact subsets in $X$ or $Y$.
\item $d_{\overline{K}}(.,.)$, or $d_{\overline{K'}}(.,.)$ stands for distance between chambers in $\partial X$ or $\partial Y$.
\item $dist(.,.)$ stands for combinatorial distance between two chambers or a face and a chamber in $\partial X$ or $\partial Y$.
\end{itemize}

In $X$ or $Y$, flats are maximal dimension isometric copies of Euclidean spaces. By hyperplane, we mean codimension 1 subflat in a flat. In $\partial X$ or $\partial Y$, an apartment is the boundary at infinity of a flat in $X$ or $Y$. A wall is the boundary of a hyperplane. A face is the boundary at infinity of a Weyl subsector in some hyperplane.

\section{Mapping flats to flats}
The $(L,C)$ QI-embedding $\varphi: \Gamma\to \Lambda$ induces a map of $X$ into $Y$, that is the composition of $\varphi$ and nearest point projection onto $\Gamma$. We also denote the resulting map  $\varphi$.

The outline of this section is: first we show that the image under $\varphi$ of any flat in a certain family is a flat in asymptotic cones. Then, taking advantage of this conclusion, we show that image of a flat is sublinearly diverging from an actual flat. Moreover, we could show then the image of a large proportion is uniformly closed to a flat. Readers may see same arguments in \cite{Dru} in a different order. The essential difference here is that we do not need the logarithm law, but only the ergodic theorem.

\subsection{Good flats}
\label{subsection:goodflats}

We now start with constructing a family of flats on which $\varphi$ behaves well.

Let $x_0=p(1)\in \Gamma\backslash G$, let $p(d)$ be a number so that the volume of the ball center $B(x_0,d)$ is $1-p(d)$. Note that the $d$-neighborhood of $\Gamma$, denoted $\Nbhd_d(\Gamma)$, is $p^{-1}(B(x_0,d))$. For $x,y\in \Nbhd_d(\Gamma)$:
\[L^{-1}d(x,y)-C-2L^{-1}d<d(\varphi(x),\varphi(y))<Ld(x,y)+C+2Ld.\]
By the ergodic theorem, for a.e. $g\in G$, for any $\sigma\subset \Xi$:
\[\liminf\limits_{r\to\infty}\frac{|F_\sigma\cap \Nbhd_d(\Gamma)\cap B(o,r)|}{|F_\sigma\cap B(o,r)|} > 1- 2p(d),\]
where $F_\sigma=\pi(gA_\sigma)$, $o=\pi(g)$, and $|.|$ stands for appropriate (dimension) Lebesgue measure in Euclidean (sub)flats. Let
\[r'(g,d)=\inf\{ s>0: \frac{|F_\sigma\cap \Nbhd_d(\Gamma)\cap B(o,r)|}{|F_\sigma\cap B(o,r)|} > 1- 2p(d), \forall \sigma\subset \Xi,\forall r>s\}.\]
By ergodic theorem, for any $d$, $r'(g,d)<+\infty$ for a.e. $g\in \Gamma\backslash G$. For every $d$, set $\Omega '(R,d)=\{g\in \Gamma\backslash G: r'(g,d)<R\}$. Then $\lim\limits_{R\to\infty}\mu(\Omega'(R,d))=1$ for all $d$.

Fix a $\delta>0$, there is an increasing sequence $(R_d)$ such that $R_d>e^d$ and $\mu(\Omega '(R_d,d))>1-\frac{\delta}{|\Xi|2^{d+2}}$. Set
\[\Omega '_\delta=\cap_{\alpha\in\Xi\cup\{\varnothing\}}\cap_{d=1}^\infty \Omega '(R_d,d)k_\alpha.\]
We have that $\mu(\Omega_\delta ')>1-\frac{\delta}{2}$. If $g\in \Omega_\delta '$, consider the finite union of flats $\pi(gA)\cup (\cup_{\alpha\in\Xi} \pi(gk_\alpha A))$. Note that $\cup_{\alpha\in\Xi} \pi(gk_\alpha A)$ intersects with flat $\pi(gA)$ in the Weyl pattern at $\pi(g)$. If $x\in \pi(gA)\cup (\cup_{\alpha\in\Xi} \pi(gk_\alpha A))$, then $x$ is \\$\big(\log(d(o,x)+1)+ 2p(\log(d(o,x)+1)d(o,x)\big)$-close to $\Gamma$. Therefore, for any\\ $x,y\in \pi(gA)\cup (\cup_{\alpha\in\Xi} \pi(gk_\alpha A))$, we can estimate:
\begin{multline}\label{distanceformula}
L^{-1}d(x,y)-C-L^{-1}\log(d(o,x)+1)-L^{-1}\log(d(o,y)+1)\\-2L^{-1}p(\log(d(o,x)+1))d(o,x) -2L^{-1}p(\log(d(o,y)+1))d(o,y)<d(\varphi(x),\varphi(y))<\\Ld(x,y)+C+L\log(d(o,x)+1)+L\log(d(o,y)+1)\\+2Lp(\log(d(o,x)+1))d(o,x) +2Lp(\log(d(o,y)+1))d(o,y).
\end{multline}
If we set $\beta(s)=\frac{\log(s+1)}{s}+2p(\log(s+1))$, then $\beta$ is decreasing to 0 on $[0,+\infty)$. Then (\ref{distanceformula}) can be rewritten as:
\begin{multline}\label{refinedistance}
L^{-1}d(x,y)-C-L^{-1}\beta(d(o,x))d(o,x)-L^{-1}\beta(d(o,y))d(o,y)<d(\varphi(x),\varphi(y))\\<Ld(x,y)+C+L\beta(d(o,x))d(o,x)+L\beta(d(o,y))d(o,y).
\end{multline}
This seems complicated but note that there are only two linear growing terms, $L^{-1}d(x,y)$ and $Ld(x,y)$. Other terms are sublinear and will disappear when we take asymptotic cones.

Repeat the argument in order to obtain a refined family of flats as follows. Set
\[r(g,\delta)=\inf\{s>0: \frac{|F\cap \Gamma \Omega_\delta '\cap B(o,r)|}{|F\cap B(o,r)|}>\delta, \forall r>s\},\]
where $F=\pi(gA), o=\pi(g)$. Note that for any $\delta>0$, $r(g,\delta)<+\infty$ for a.e. $g\in \Gamma\backslash G$.  Then set
\[\Omega(R, \delta)=\{g\in\Omega_\delta ': r(g,\delta)<R\}.\]
There exists $R(\delta)$ such that $\mu(\Omega(R(\delta),\delta))>\frac{\delta}{2}$. Moreover, $R(\delta)$ could be chosen to be non-increasing.
Set
\[\Omega_\delta=\cap_{k=1}^\infty \Omega(R(\frac{\delta}{2^{k-1}}),\frac{\delta}{2^{k-1}}).\]
and $\theta_\delta: (0,\infty)\to [0,1]$ be a function defined by $\theta_\delta (s)=\frac{\delta}{2^{k}}$ if $R(\frac{\delta}{2^{k-1}})<s\le R(\frac{\delta}{2^k})$ for $k=1,2,\dots$, and $\theta_\delta(s)=\delta$ is $0<s<R(\delta)$.\\
Now if $g\in \Omega_\delta$ then we get an estimate \eqref{refinedistance} about the distance under $\varphi$ of finite union of transverse flats $\cup_{\alpha\in \Xi\cup\{\varnothing\}}\pi(gk_\alpha A)$ through $\pi(g)$. Moreover, if $z\in\pi(gA)$ then there is a point $x\in \pi(gA)$ at most distance $\theta_\delta(d(o,z))d(o,z)$ from $z$, such that we have the estimate of the distance for image under $\varphi$ of a finite union of transverse flats through $x$.

Terminology: we say a flat $F$ is sub-$\theta_\delta$-diverging w.r.t $x$ if $F=\pi(gA)$ for some $g\in\Omega_\delta$ such that $x=\pi(g)$.

By the ergodic theorem, for almost every $g\in G$, for any $\alpha\in \varnothing\cup \Xi$, in (sub)flat $F_\alpha=\pi(gA_\alpha)$ we have
\[\liminf\limits_{r\to\infty}\frac{|\{v\in A_\alpha\cap B_r: gv\in \Gamma\Omega_\delta\}|}{|B_r|}> 1-2\delta,\]
where $B_r$ denotes the Euclidean ball in appropriate dimension, centered at origin, radius $r$. Let $\mathcal{G}$ be the full measure subset consisting of such $g\in G$. Note that we can take $\mathcal{G}$ to be $\Gamma$-invariant by defining $\mathcal{G}$ as the disjoint union of  the $\Gamma$ translates  of a full measure subset in $\Gamma\backslash G$. Let $\mathcal{F}$ be the family of flats that have form $F=\pi(gA)$ for some $g\in\mathcal{G}$. This say that if a flat $F$ is in the family $\mathcal{F}$, then $F$ is sub-$\theta_\delta$-diverging w.r.t. a large portion of points in $F$, also w.r.t. a large portion of points in a finite union of certain hyperplanes.

\subsection*{Taking asymptotic cones.}
We denote $[x_n]$ the point in a asymptotic cone represented by the sequence $(x_n)$. For a sequence of sets $(D_n)$, similarly we denote $[D_n]$ be a subset of a asymptotic cone consisting of points $[x_n]$, where $x_n\in D_n$ for every $n$. And we denote $[x]$, $[D]$ for the case $x_n=x$, $D_n=D$ for all $n$.

We will show that if $F_n$ is a sub-$\theta_\delta$-diverging flat w.r.t. $x_n$ then restriction of $\varphi$ on $F_n$ induces a biLipschitz map from a flat $[F_n] \subset \Cone(X,x_n,c_n,\omega)$ into $\Cone(Y,y_n,c_n,\omega)$, where $y_n=\varphi(x_n)$, $\omega$ is arbitrary nonprincipal ultrafilter, and $(c_n)$ is any sequence $\omega$-converging to infinity.
Indeed, if $(u_n)$, $(v_n)$ are two sequences in $F$ such that $\lim\limits_{\omega}\frac{d(x_n,u_n)}{c_n}=d_1<+\infty$ and $\lim\limits_{\omega}\frac{d(x_n,v_n)}{c_n}=d_2<\infty$ then
\begin{align*}
\frac{d(\varphi(x_n),\varphi(u_n))}{c_n}<L\frac{d(x_n,u_n)}{c_n}+\frac{C}{c_n}+L\frac{\beta(d(x_n,u_n))d(x_n,u_n)}{c_n}.\end{align*}
We see that if $\lim\limits_\omega\frac{d(x_n,u_n)}{c_n}>0$ then $\lim\limits_\omega \beta(d(x_n,u_n))=0$, thus we always have $\lim\limits_\omega \frac{\beta(d(x_n,u_n))d(x_n,u_n)}{c_n}=0$. So $[\varphi(u_n)]$ represents a point in $\Cone(Y,y_n,c_n,\omega)$. Moreover,
\begin{align*}L^{-1}\frac{d(u_n,v_n)}{c_n}-&\frac{C}{c_n}-L^{-1}\frac{\beta(d(x_n,u_n))d(x_n,u_n)}{c_n}-L^{-1}\frac{\beta(d(x_n,v_n))d(x_n,v_n)}{c_n}\\
<&\frac{d(\varphi(u_n),\varphi(v_n))}{c_n}<\\
<L\frac{d(u_n,v_n)}{c_n}+&\frac{C}{c_n}+L\frac{\beta(d(x_n,u_n))d(x_n,u_n)}{c_n}+L\frac{\beta(d(x_n,v_n))d(x_n,v_n)}{c_n}.\end{align*}
As above, we can see that all terms, possibly except $L\frac{d(u_n,v_n)}{c_n}$ have $\omega$-limits are zero. From this $\varphi$ induced a well-defined map $[\varphi]$ on $[[F_n]$. Moreover, this map is $L$-biLipschitz.

Assume $u_\omega=[u_n]$ be an arbitrary point in the flat $[F_n]$. We show that, not only there is a biLipschitz map on $[F_n]$ but also at any point $u_\omega\in [F_n]$, there is a $L$-biLipschitz map on finite union of flats intersecting $[F_n]$ at the Weyl pattern at $u_\omega$, and agrees with $[\varphi]$ on $[F_n]$. Abusing notation, we shall still denote the induced map on finite union of flats $[\varphi]$

In case $u_\omega=[x_n]$ the we will pick the sequence of finite unions of flats going through $x_n$ as follows: let $g_n\in\Omega_\delta$ such that $F_n=\pi(g_nA)$ and $x_n=\pi(g_n)$. For each $\alpha\in \Xi\cup\{\varnothing\}$, denote $F_{n,\alpha}=\pi(g_n k_\alpha A)$. We still have the estimate \eqref{refinedistance} for image under $\varphi$ of the finite union of flats $\cup_{\alpha\in \Xi\cup\{\varnothing\}}F_{n,\alpha}$. Therefore, $\varphi$ induces a $L$-biLipschitz map on $[\cup_{\alpha\in\Xi\cup\{\varnothing\}}F_{n,\alpha}]=\cup_{\alpha\in\Xi\cup\{\varnothing\}}[F_{n,\alpha}]$. Note that $[F_{n,\alpha}]$ intersects $[F_n]$ exactly at hyperplane containing $x_\omega=u_\omega$.

In the case $d_\omega(x_\omega,u_\omega)=d>0$, then $\lim\limits_\omega d(x_n,u_n)=+\infty$. Since $x_n=\pi(g_n)$ where $g_n\in\Omega_\delta$, by the definition of $\Omega_\delta$, there is $v_n=\pi(h_n)\in F_n$ such that
\begin{itemize}
\item $F_n=\pi(h_nA)$.
\item $d(u_n,v_n)<\theta(d(x_n,u_n))d(x_n,u_n)$, thus $[u_n]=[v_n]$.
\item at each $v_n$, there is a finite union of flats $\cup_{\alpha\in\Xi} F_{n,\alpha}$, where $F_{n,\alpha}=\pi(h_nk_\alpha A)$, containing $v_n$ and intersects $F_n$ exactly at hyperplanes going through $v_n$. Moreover, the estimate \eqref{refinedistance} works for each of the finite union of flats $F_n\cup (\cup_{\alpha\in\Xi}F_{n,\alpha})$, where $v_n$ have role of center $o$ in this situation.
\end{itemize}
Note that any sequence $(w_n)$ with $\lim\limits_\omega \frac{d(w_n,x_n)}{c_n}<+\infty$ also satisfy $\lim\limits_\omega\frac{d(w_n,v_n)}{c_n}<\infty$. Therefore $\varphi$ induces a well-defined biLipschitz map $[\varphi]$ on $[F_n]\cup_{\alpha\in\Xi}[F_{n,\alpha}]$. It is easy to see that each $[F_{n,\alpha}]$ intersects $[F]$ at the Weyl pattern at $[v_n]=[u_n]$. Note that different choices of $(v_n)$ resulted in different finite union of flats in the asymptotic cone. However, each finite union of flats always intersect with flat $[F_n]$ at the Weyl pattern at $[u_n]$.

\subsection{$[\varphi]([F_n])$ is a flat in $ \Cone (Y,y,c_n,\omega)$.}
\label{subsection:flatstoflats}

The idea of argument here is the same as in section 3.2 in \cite{FiWh}.

There exists a finite union of codimension 2 hyperplanes in $[F_n]$ such that on its complement, $\Sigma$, $[\varphi]$ locally maps into a flat. Since at each point in $[F_n]$, there are transverse flats, and a biLipschitz map defined on the finite union of flats that agrees with $[\varphi]$ on $[F_n]$, we can deduce that $[\varphi]$ locally maps Weyl pattern to Weyl pattern. Also, $[\varphi]$ is biLipschitz, hence differentiable almost everywhere. At points of differentiability in $\Sigma$, $D[\varphi]$ is a linear map preserving Weyl pattern. By assumption, that linear map is conformal. This implies locally $[\varphi]$ is 1-quasiconformal a.e.

So locally $[\varphi]$ is a quasi-conformal map that is 1-quasi-conformal a.e., by Gehring's theorem, $[\varphi]$ is smooth. So $[\varphi]$ has a derivative, and the derivative is continuous everywhere in $\Sigma$.

Consider $z_\omega\in \Sigma$, and $[\varphi]$ maps a connected neighborhood $U$ of $z_\omega$ into a flat. Choose some coordinate for $U$ and $[\varphi](U)$. Derivative $D[\varphi]$ is 1-quasiconformal linear map that preserves Weyl pattern at each point. Therefore, $D[\varphi]$ at each point is a composition of a constant multiple of identity and a linear Weyl element in the Weyl group associated with the symmetric space $Y$. Since $D[\varphi]$ is continuous, the Weyl elements component of the derivatives are the same for all points in $U$. So, up to composing with an element of Weyl group, we can assume that the derivative at each point is a multiple of identity. In the chosen coordinate, $D[\varphi](v_\omega)=f(v_\omega)\cdot Id$, for all $v_\omega\in U$. This implies we can write $[\varphi](v_\omega)=([\varphi]_{1}(v_\omega ),\ldots,[\varphi]_{d}(v_\omega))$. Then $\frac{\partial}{\partial v_{\omega ,i}}[\varphi]_{j}(v_\omega)=0$ for $1\le i\ne j\le d$. Hence, $[\varphi]_{i}$ only depends on $v_{\omega,i}$, i.e.
\[[\varphi](v_\omega)=([\varphi]_{1}(v_{\omega,1}),\ldots,[\varphi]_{d}(v_{\omega,d})).\]

$D[\varphi](v_\omega)=f(v_\omega)\cdot Id$ would imply that $[\varphi]_{1}'(v_{\omega,1})=\ldots= [\varphi]_{d}'(v_{\omega,d})=f(v_\omega)$ for all $v_\omega\in U$. Therefore $f(v_\omega+(0,\dots,\epsilon,\dots,0))=\varphi_{\omega,1}'(v_{\omega,1})=f(v_\omega)$. This implies $f$ is constant in $U$. Hence $[\varphi]$ is a fixed constant multiple of identity on the whole $U$.

Now, we know $[\varphi]$ locally is a composition of multiple of identity and an element of Weyl group. $[\varphi]$ is continuous on $\Sigma$, and $\Sigma$ is connected, thus on $\Sigma$, $[\varphi]$ has to be a fixed constant multiple of identity up to composing with a unique element of Weyl group. This property of $[\varphi]$ has to be true every where on the flat $[F_n]$ too, because $[\varphi]$ is continuous on the flat and $\Sigma$ is the complement of a codimension 2 set. Therefore $[\varphi](F_\omega)$ is a flat in $\Cone(Y,\varphi(x_n),c_n,\omega)$.

To summarize, what we have proved is the following:

\begin{proposition}\label{isoflat}
For any sequence $(c_n)$ with $\lim\limits_\omega c_n =\infty$ and any sequence of flats $F_n$ which is sub-$\theta_\delta$-diverging w.r.t. $x_n$, then $[\varphi(F_n)]$ is a flat in $\Cone(X,\varphi(x_n), c_n,\omega)$. Moreover, $\varphi$ induces a scalar multiple of an isometry $[\varphi]$ from flat $[F_n]$ in \\$\Cone(X,x_n,c_n,\omega)$ to flat $[\varphi(F_n)]$ in $\Cone(Y,\varphi(x_n),c_n,\omega)$.
\end{proposition}

\begin{remark}
 The proposition above implies that for a flat $F\in\mathcal{F}$, $[\varphi(F)]$ is a flat in asymptotic cones with arbitrary rescaling sequence which have $\omega$-limit infinity. However, the sequence of based points are not arbitrary. Otherwise, by proposition 7.1.1 in \cite{KL2}, $\varphi(F)$ is uniformly close to a flat. And this cannot be expected in the case $\varphi$ is induced from a QI-embedding of a nonuniform lattice.
\end{remark}

\subsection{Associating to  $\varphi(F)$ a unique flat in $Y$}

\begin{proposition}\label{1flatprop}
Let $F\in\mathcal{F}$ be a flat that is sub-$\theta_\delta$-diverging w.r.t. $x\in F$. Then there is a flat $F'\subset Y$ such that for any sequence $(c_n)$ that has $\lim\limits_\omega c_n=\infty$, in \\$\Cone(Y, \varphi(x),c_n,\omega)$:
\[[\varphi(F)]=[F'].\]
Moreover, the flat $F'$ does not depend on which point $x\in F$ that we choose. This implies that $F'$ is unique.\end{proposition}

Let $W_\omega^1,\cdots,W_\omega^p$ be distinct Weyl chambers at $[y]=[\varphi(x)]=[\varphi]([x])$ such that the flat $[\varphi]([F])=\bigcup_{j=1}^p W_\omega^j$. For each $W_\omega^j$, there is sequence of Weyl chambers $(W_n^j)$ such that $W_\omega^j=[W_n^j]$.

For $\epsilon, \rho >0$, we denote $C_\epsilon (y,\rho)=\{u\in Y: (1-\epsilon)\rho<d(y,u)<(1+\epsilon)\rho\}$, and $S(y,\rho)=\{u\in Y: d(y,u)=\rho\}$. We prove a lemma estimating the divergence away from chambers.

\begin{lemma}\label{quantifylemma}
Let $U$ be a subset of $Y$, containing $y$, such that in any asymptotic cone $\Cone(Y,y,d_n,\omega)$, $[U]$ is always a flat for any rescaling sequence $(d_n)$ with $\lim\limits_\omega d_n=\infty$. Assume in $\Cone(Y,y,c_n,\omega)$, we have $[U]=F_\omega$, where flat $F_\omega=\cup_{j=1}^pW_\omega^j$, union of Weyl chambers vertex at $[y]$. Let $W_n^j$ be Weyl chambers vertex at $y$ such that $W_\omega^j=[W_n^j]$, for all $j=1,\cdots, p$. For every $\epsilon>0$,  there exists $ R_\epsilon$ such that for $\omega$-a.e. $n$, and for all $\rho\in [R_\epsilon,c_n]$:
\[\sup_{z\in S(y,\rho)\cap U}d(z,\cup_{j=1}^p W_n^j)< \epsilon\rho.\]
\end{lemma}
\begin{proof}
Let
\[R_{\epsilon,n}=\sup\{\rho\in[1,c_n]:\sup_{z\in S(y,\rho)\cap U}d(z,W_n^j)\ge\epsilon\rho\}.\]
We need to show $\lim\limits_\omega R_{\epsilon,n} < +\infty$ for every $\epsilon >0$.

If  $\lim\limits_\omega \frac{R_{\epsilon,n}}{c_n}= \sigma >0$ then take $z_n\in S(y,R_{\epsilon,n})\cap U$ such that
\[d(z_n, \cup_{j=1}^p W_n^j) \ge \epsilon R_{\epsilon,n}.\]
\[\implies d_\omega([z_n], \cup_{j=1}^p W_\omega^j)\ge \epsilon \sigma.\]
This contradicts that $[U]=F_\omega=\cup_{j=1}^p W_\omega^j$.

Therefore $\lim\limits_\omega \frac{R_{\epsilon,n}}{c_n}=0$. Suppose that $\lim\limits_\omega R_{\epsilon,n}=+\infty$, in $\Cone(Y,y,R_{\epsilon,n},\omega)$, $[W_n^j]$ is also a Weyl chamber for all $j=1,\cdots,p$. Let $z_n$ as above, we have
\[d_\omega ([y],[z_n])=1\]
and
\[d_\omega([z_n],\cup_{j=1}^p [W_n^j])\ge \epsilon.\]
Since $\lim\limits_\omega \frac{R_{\epsilon,n}}{c_n} =0$, $2R_{\epsilon,n}<c_n$ for $\omega$-a.e. $n$. And by definition of $R_{\epsilon,n}$, for all $u_n\in U\cap C_\frac{\epsilon}{100}(y,2R_{\epsilon,n})$, we have $d(u_n,\cup_{j=1}^p W_n^j)< \epsilon d(y,u_n)$. Then $u_\omega =[u_n]$ has properties
\[2-\frac{\epsilon}{100}\le d(y_\omega, u_\omega) \le 2 +\frac{\epsilon}{100},\]
\[d(u_\omega,\cup_{j=1}^p W_\omega^j)<\epsilon.\]
The point $z_\omega$ is in the flat $[U]$. Let $z_\omega ' \in [U]$ on the ray $\overrightarrow{y_\omega z_\omega}$ that has $d(y_\omega,z_\omega')\in\\ (2-\frac{\epsilon}{100},2+\frac{\epsilon}{100})$. Obviously, $z_\omega '$ has to be a limit of such a sequence $u_n$ above. On the ray $\overrightarrow{y_\omega z_\omega '} $, we have
\[d(z_\omega, \cup_{j=1}^p W_\omega^j)\ge \epsilon,\]
\[d(z_\omega ', \cup_{j=1}^p W_\omega^j) < \epsilon.\]
This contradicts the convexity of the metric  on symmetric spaces. Therefore, the lemma is proven.
\end{proof}
\begin{remark}
Instead of working with a fixed $y$ and $U$, we could also make the assumption that $y_n\in U_n$ such that $[U_n]$ is a flat in any asymptotic cone.

We could also drop the assumption that $[U]$ is a flat in any asymptotic cones, and only need that $[U]$ is the limit of sequence of finite union of Weyl chambers in $\Cone(Y,y,c_n,\omega)$. The proof will be similar, but requires some extra technical details. See \cite[Lemma 3.3.5]{Dru}.
\end{remark}

For any sequence $(c_n)$ with $\lim\limits_\omega c_n= \infty$, by previous subsection, $[\varphi] ([F])$ is a flat in $\Cone(Y,y,c_n,\omega)$. Let $[\varphi]([F])=\cup_{j=1}^pW_\omega^j=\cup_{j=1}^p[W_n^j]$, where $W_\omega^j$, $W_n^j$ are Weyl chambers with vertex at, respectively, $[y]$ and $y$. Let $W^j$ be the Gromov-Hausdorff limit of the $W_n^j$, i.e. for any $r>0$, $W_n^j\cap B(y,r)\to W^j\cap B(y,r)$ in Hausdorff metric.

For $i\neq j$ let $u_n^i\in W_n^i$, $u_n^j\in W_n^j$ such that $[\overrightarrow{yu_n^i}]$ and $[\overrightarrow{yu_n^j}]$ are two rays in in $W_\omega^i$ and $W_\omega^j$. If two rays coincide then they also coincide in $\Cone(Y,y,1,\omega)$. The limit of rays in this cone is exactly the Gromov-Hausforff limit. This implies $W^i$ and $W^j$ have the same or more adjacency relation as the adjacency relation of $W_\omega^i$ and $W_\omega^j$. Because $\cup_{j=1}^pW_\omega^j(\infty) $ is a sphere, we know that $\cup_{j=1}^pW^j(\infty)$ is a Lipschitz sphere. Here, by sphere we mean an apartment at infinity. In order to show $\cup_{j=1}^pW_\omega^j(\infty) $ is actually a sphere, we show that the adjacency relation of $\{W^j(\infty): j=1,\dots,p\}$ is just actually same as adjacency relation of $\{W^j_\omega(\infty): j=1,\dots,p\}$.

Fixing some $\lambda$, we let
\[R_n=\sup\{\rho \in (0,c_n] : W_n^j\cap B(y,\rho)\subset \Nbhd_\lambda(W^j)\forall j\in\{1,\dots,p\}\}.\]
Then we always have $\lim\limits_\omega R_n=\infty$ because of the definition of $W^j$.

We claim that in $\Cone(Y,y,\sqrt{R_n},\omega)$ we have
\begin{equation}\label{claim1}[\varphi (F)] \subset\cup_{j=1}^p[W_n^j]=\cup_{j=1}^p[W^j].\end{equation}
The last equality is obvious, we only need to prove the inclusion. Let $z_\omega =[z_n]\in [\varphi (F)]$ where $z_n\in \varphi (F)\cap C_{\zeta_n}(y,\eta\sqrt{R_n})$, where $\lim\limits_\omega \zeta_n =0$. Then $d(y_\omega,z_\omega)=\eta$. By Lemma \ref{quantifylemma}, for any $\epsilon >0$, there exists $ R_\epsilon$ such that $\omega$-a.e. $n$
\[\sup_{z\in S(y,\rho)\cap \varphi(F)} d(z,\cup_{j=1}^p W_n^j)< \epsilon\rho,\]
for all $\rho\in[R_\epsilon,c_n)$. Since $\omega$-a.e. $n$, we have $R_\epsilon < d(y,z_n) <  c_n$, we must have
\[d(z_n,\cup_{j=1}^p W_n^j ) < \epsilon \eta (1+\zeta_n)\sqrt{R_n}.\]
Thus,
\[d(z_\omega ,\cup_{j=1}^p [W_n^j]) < \epsilon\eta.\]
Since $\epsilon$ can be arbitrarily small, $z_\omega \in \cup_{j=1}^p [W_n^j]$. This proves the claim, i.e. in $\Cone(Y,y,\sqrt{R_n},\omega)$, $[\varphi(F)]\subset \cup_{j=1}^p[W^j]$.

We also know that $[\varphi(F)]$ is a flat. On the other hand, $\cup_{j=1}^p[W^j]$ is a union of Weyl chambers at a common vertex, by the claim, containing a flat through that vertex. Moreover that the number of Weyl chambers in the union is exactly the number of a Weyl chambers we have in a sphere apartment at infinity. Hence, we have an equality rather than an inclusion, i.e. $[\varphi(F)]= \cup_{j=1}^p[W^j]$.
Applying lemma \ref{quantifylemma} to $\Cone(Y,y,\sqrt{R_n},\omega)$, $\varphi(F)$ and the family of chambers $W^j$, we get

\begin{corollary}\label{quancor}
For all $\epsilon >0$, there exists $R_\epsilon$, such that for $\omega$-a.e. $n$
\[\sup_{z\in S(y,\rho)\cap U} d(z, \cup_{j=1}^p W^j ) <\epsilon\rho,\]
for all $\rho\ge R_\epsilon$.
\end{corollary}

\begin{proof}
Apply the lemma \ref{quantifylemma}, for $\omega$-a.e. $n$
\[\sup_{z\in S(y,\rho)\cap U} d(z, \cup_{j=1}^p W^j ) \le \epsilon\rho,\]
for all $\rho\in [R_\epsilon,\sqrt{R_n})$. Since $\lim\limits_\omega\sqrt{R_n}=\infty$, the estimate holds for all $\rho\ge R_\epsilon$.
\end{proof}

\begin{corollary}\label{flatWeylcor}
For any sequence $(c_n)$ such that $\lim\limits_\omega c_n =\infty$, in $\Cone(Y,y,c_n,\omega)$
\[[\varphi(F)]=\cup_{j=1}^p [W^j].\]
\end{corollary}
\begin{proof}
By the previous corollary, for all $\epsilon >0$, there exists $ R_\epsilon$ such that $d(z,\cup_{j=1}^pW^j) \\< \epsilon d(y,z)$, for $d(y,z)>R_\epsilon$.
This implies that for all $z_\omega =[z_n]\in [\varphi(F)]$, and for all $\epsilon>0$: $d(z_\omega, \cup_{j=1}^p [W^j]) < \epsilon$. Thus, $[\varphi(F)]\subset \cup_{j=1}^p [W^j]$.

 As before, $\cup_{j=1}^p [W^j]$ is a set of finite union of Weyl chambers at the same vertex, and the number of the chambers is exactly the number of Weyl chambers that a flat can have. Moreover, $[\varphi(F)]$ is a flat, containing the vertex of chambers. This implies the union of Weyl chambers is exactly the flat, i.e. $[\varphi(F)] = \cup_{j=1}^p[W^j]$.
\end{proof}

\begin{proof}[Proof of Proposition \ref{1flatprop}]

We use same notations as in Lemma \ref{quantifylemma}, Corollary \ref{quancor}, and \ref{flatWeylcor}. As the result of Corollary \ref{flatWeylcor}, there are Weyl chambers $W^j$ in $Y$ vertex at $y$ for $j=1,\dots,p$ such that $[\varphi(F)]=\cup_{j=1}^p[W^j]$.

Since any two rays in $Y$ are either asymptotic or diverging linearly, and for any ray $r_\omega \subset [W^j]$ there is a ray $r\subset W^j$ such that $r_\omega =[r]$, we can conclude that $\{W^j(\infty)\}_{j=1}^p$ has the same adjacency relation as  $\{[W^j](\infty)\}_{j=1}^p$; Therefore $\cup_{j=1}^p W^j(\infty)$ is a biLischitz sphere in Tits metric, and containing exactly the same number of chambers as in an apartment. Hence $\cup_{j=1}^p W^j(\infty)$ is an apartment in $\partial Y$. Thus there is an apartment $F_y'\subset Y$ such that $F'(\infty) =\cup_{j=1}^p W^j(\infty)$. It follows easily that $[F']=[\cup_{j=1}^pW^j]=[\varphi(F)]$.

We now prove uniqueness. Let $x_1,x_2\in F$ such that the flat $F$ is sub-$\theta_\delta$-diverging w.r.t.~both $x_1$ and $x_2$. Note that $\Cone(Y,\varphi(x_1),c_n,\omega)$ and $\Cone(Y,\varphi(x_2),c_n,\omega)$ are canonically isometric by identity map on each sequence. So if in $\Cone(Y,\varphi(x_1),c_n,\omega)$, we have $[\varphi(F)]=\cup_{j=1}^p[W^j]$, then that equality will still hold true in $\Cone(Y,\varphi(x_2),c_n,\omega)$. Therefore the two flats $F_{\varphi(x_1)}'$ and $F_{\varphi(x_2)}'$ coincide.

So for every flat $F\in \mathcal{F}$, we can associate a unique flat $F'$ such that if $F$ is sub-$\theta_\delta$-diverging w.r.t.~$x$ then in $\Cone(Y,\varphi(x),c_n,\omega)$ we have $[\varphi(F)]=[F']$.
\end{proof}

Because of the assumption about Weyl patterns, we can treat $\partial X$ as a non-thick building with the Coxeter structure the one for $\partial Y$. So for $F$ sub-$\theta_\delta$-diverging flat w.r.t.~ $x\in F$, there are $p$ Weyl chambers $W_1,\dots,W_p$ vertex at $x$ with respect to the non-thick structure such that $F=\cup_{j=1}^p W_j$. Let $F'$ be the flat associated to $\varphi(F)$ as in Proposition \ref{1flatprop}, and $W^1,\dots,W^p$ be chambers vertex at $y$, projection of $\varphi(x)$ on $F'$, such that $F'=\cup_{j=1}^p W^j$ . By Proposition \ref{isoflat}, $[\varphi]$ isometrically map $[F]$ to $[F']$. Hence with appropriate order, we have that $[\varphi(W_j)]=[\varphi]([W_j])=[W^j]$ for all $j=1,\dots,p$.

\begin{corollary}\label{cham2cham}
For any $\epsilon>0$, there is $R_\epsilon >0$, independent of $F$, such that for all $j=1,\dots,p$, and for any $\rho>R_\epsilon$
\[\sup_{z\in \varphi(W_j)\cap S(y,\rho)}d(z,W^j)<\epsilon\rho.\]
\end{corollary}

\begin{proof}Suppose not, there is a sequence of flats $F_n$ with $F_n= \cup_{j=1}^p W_{j,n}$, where $W_{j,n}$ are chambers vertex at $x_n$; the sequence associated flats $F_n'=\cup_{j=1}^p W_n^j$ where $W_n^j$ are chambers vertex at $y_n$, the projection of $\varphi(x_n)$ on $F_n'$; and there are $R_n>0$ with $\lim\limits_{n\to\infty}R_n=\infty$, and $z_n\in \varphi(W_{j,n})\cap S(y,R_n)$ such that $d(z_n,W_n^j)\ge \epsilon R_n$. The flat $[F_n]$ in $\Cone(X,x_n,R_n,\omega)$ is mapped by a scaling factor of isometry $[\varphi]$ to the flat $[F_n']$ in $\Cone(Y,y_n,R_n,\omega)$. Thus, $[\varphi]([W_{j,n}])=[W_n^j]$. However $[z_n]\in [\varphi(W_j)]=[\varphi]([W_n^j]$ satisfies $d_\omega([z_n],[W_n^j])\ge \epsilon$. This is a contradiction.
\end{proof}

Next, we prove that $\varphi$ maps a large proportion of $F$ into a neighborhood of $F'$.

\begin{proposition}\label{Dcloseprop}
There exists $D(L,C,\delta)$ such that if $F\in\mathcal{F}$ is sub-$\theta_\delta$-diverging w.r.t.~ $x$, and $F'\subset Y$ is the flat associated with the image $\varphi (F)$, then $d(\varphi(x), F') < D$.
\end{proposition}

\begin{proof}
Suppose not: then there exist $ F_n\in\mathcal{F}$, $x_n\in F_n$ such that $F_n$ is sub-$\theta_\delta$-diverging w.r.t.~$x_n$ and $c_n=d(x_n,F_n')\to \infty$ as $n\to \infty$, where $F_n'$ is flat in $Y$ associated to $\varphi(F_n)$ by Proposition \ref{1flatprop}. Denote $y_n =\varphi(x_n)$. Consider $[\varphi(F_n)]\subset \Cone (Y,y_n,c_n,\omega)$. Let $\cup_{j=1}^p W_n^j$ be the union of Weyl chambers vertices at $y_n$ such that $F_n'(\infty)=\cup_{j=1}^p W_n^j(\infty)$.

By Corollary \ref{quancor}, for any $\epsilon >0$, there is $R_\epsilon$ such that
\[\sup_{z\in C _\frac{\epsilon}{100}(y_n,\rho)}d(z, \cup_{j=1}^pW_n^j)\le 5\epsilon\rho,\]
for all $\rho\ge R_\epsilon$. Thus $[\varphi(F_n)]\subset\cup_{j=1}^p [W_n^j]$. But $[\varphi(F_n)]$ is a flat, and by the argument before, we get the equality $[\varphi(F_n)]=\cup_{j=1}^p [W_n^j]$. Thus $\cup_{j=1}^p [W_n^j]$ is a flat.

We have $d_{Hau}(W_n^j,F_n')=d(y_n,F_n')=c_n$ for all $j=1,\cdots,p$, implies that $d_{Hau}(\cup_{j=1}^p [W_n^j], [F_n'])=1$. Note that $[F_n']$ is also a flat. So we have two flats have Hausdorff distance 1 from each other. This is a contradiction.
\end{proof}

\section{Measurable boundary map and continuity on stars}
\label{section:firstcontinuity}

So far, we have associated to each flat $F\in\mathcal{F}$ a flat $F'$ in $Y$. We now want to consider the correspondence at the level of Weyl chambers. Let $y=\varphi(x)$, consider the map $[\varphi] |_{F_\omega}: F_\omega\to F_\omega'\subset \Cone(Y,y,c_n,\omega)$. Up to rescaling a factor, $[\varphi] |_{F_\omega}$ is an isometry preserving the Weyl chamber pattern. So $[\varphi] |_{F_\omega}$ maps each Weyl chamber to a finite union of chambers in $F'_\omega$. There is an obvious correspondence between Weyl chambers in $F_\omega(\infty)$ (respectively $F'_\omega(\infty)$) and Weyl chambers in $F(\infty)$ (respectively $F'(\infty)$). Therefore, $\varphi$ associates each chamber in $F(\infty)$ with a finite union of chambers in $F'(\infty)\subset \partial Y$.

Let $\overline{\Omega}$ be the set of Weyl chambers at infinity of flats in $\mathcal{F}$. Then $\overline{\Omega}\cap \overline{K}$ has full measure in $\overline{K}$, where $\overline{K}$ is the Furstenberg boundary of $X$.

Let $W\in \overline{\Omega}\cap \overline{K}$. If $F_1, F_2\in\mathcal{F}$ are two flats that contain $W$ in their boundaries, i.e. $W\subset F_1(\infty)$, $W\subset F_2(\infty)$. There exists $F_1', F_2'$ such that
\[[\varphi] ([F_1]) =[\varphi(F_1)]=[F_1'],\]
\[[\varphi] ([F_2]) =[\varphi(F_2)]=[F_2'].\]
Note that the map $\varphi |_{F_1\cup F_2}$ yields a well defined, biLipschitz map $[\varphi]$ on $[F_1]\cup [F_2]$. $[F_1]\cap[F_2]$ contains a Weyl chamber sector corresponding to $W$. The image of the sector under $[\varphi]$ is a finite union of chambers in $[F_1']\cap[F_2']$. Therefore, the corresponding Weyl chambers at infinity in $F_1'(\infty)$ and $F_2'(\infty)$ coincide. So we can set a correspondence
\[\overline{\varphi}(W)=(W_1', \cdots W_l').\]
In order to get a consistent way to map a Weyl chamber to a finite union of chambers, we will do as following: let $W_1$ be an arbitrary Weyl chamber in $E\cap \overline{K}$. There is a chamber $W_2\in \overline{\Omega}\cap \overline{K}$ such that there exist $ F_1,F_2\in \mathcal{F}$, and $W,W_2\subset F_1(\infty)$, $W_1,W_2\subset F_2(\infty)$. Note $[\varphi] |_{[F_1]\cup [F_2]}$ is an isometry up to a rescaling factor. There is a composition of reflections in walls of $[F_1](\infty)\cup[F_2](\infty)$ that carries $W_1$ to $W$. So there is a corresponding composition of reflections in $[F_1'](\infty)\cup [F_2'](\infty)$ carries a finite union of Weyl chambers corresponding to image of chamber $W_1$ to the finite union of chambers $W_1'\cup\dots\cup W_l'$ corresponding to image of $W$. Thinking of this as a way to label $\{1,\dots,l\}$ to finite union of chambers in the image of each chamber so that the the labeling is invariant under the induced action of Coxeter group for $\partial X$ on building $\partial Y$. Therefore, we can define a map $\overline{\varphi}$ on $\overline{\Omega}\cap \overline{K}$:
\[\overline{\varphi}: \overline{\Omega}\cap \overline{K} \to \overline{K'}\times \cdots\times \overline{K'}\]
sending each Weyl chamber to a $l$-tuple of Weyl chambers in a consistent way. We will assume that $\overline{\varphi}=(\varphi_1,\cdots,\varphi_l)$.

The rest of this section is for proving uniform continuity of $\overline{\varphi}$ on star chambers. For this we need a version of proposition \ref{Dcloseprop} for hyperplanes.

\begin{proposition}\label{hyperprop}
For a hyperplane $P=\pi(gA_\alpha)$, $g\in \Omega_\delta$, and $\alpha\in \Xi$,  there is a hyperplane $P'$ in $Y$ such that
\begin{itemize}
\item in $\Cone(Y,y,c_n,\omega)$: $[\varphi (P)]=[P']$.
\item there exists $ d(\delta,L,C,X,\Gamma)$ such that for any $z=\pi(u)$. where $u\in gA_\alpha \cap \Omega_\delta$, we have $d(\varphi(u), P')<d$. Here $d$ is independent with $P$.
\end{itemize}
\end{proposition}

\begin{proof}
Let $F_1=\pi(gA)$, $F_2=\pi(gk_\alpha A)$ Then $F_1, F_2\in \mathcal{F}$, $F_1\cap F_2 =P$,  and there is $c$ such that $\Nbhd_1(F_1)\cap \Nbhd_1(F_2)\subset \Nbhd_c(P)$.
Then there exist $F_1',F_2'\subset Y$ such that $[\varphi(F_1)] = [F_1']$ and $[\varphi(F_2)]=[F_2']$ in $\Cone(Y,y,c_n,\omega)$. $[F_1]$ and $F_2]$ are two flats whose intersection is exactly $[P]$. $[\varphi] |_{[F_1]\cup [F_2]}$ is a biLipschitz map. Thus $[F_1']\cap [F_2']$ is exactly a codimension $1$ hyperplane. Also notice that $d(F_1',F_2') < 2D$ since there is $x\in P$ such that $d(F_i',\varphi(x))<D$. This implies there exists a hyperplane $P'\subset Y$, and there exists $ d>0$ such that $P'\subset \Nbhd_D(F_1')\cap \Nbhd_D(F_2')\subset \Nbhd_d(P')$. It follows that $[P']=[\varphi(P)]$. And for $z=\pi(u)$ with $u\in gA_\alpha\cap\Omega_\delta$ then $\varphi(z)\in \Nbhd_D(F_1') \cap \Nbhd_D(F_2') \subset \Nbhd_d(P')$.

We need to show that $d$ is bounded, and does not depend on $P$. Suppose not, then there exist $P_n=\pi(g_nA_\alpha)$, $F_{1,n}=\pi(g_nA)$, $F_{2,n}=\pi(g_nk_\alpha A)$, and there is $d_n\to\infty$ such that  $\Nbhd_{d_n}(P_n')\subset \Nbhd_D(F_{1,n}')\cap \Nbhd_D(F_{2,n}')$, where $g_n\in\Omega_\delta$, and $P_n', F_{1,n}',F_{2,n}'$ are hyperplanes associated to $P_n, F_{1,n},F_{2,n}$. Let $x_n=\pi(g_n)$. In $\Cone(X,x_n, d_n,\omega)$, $[F_{1,n}]\cap[F_{2,n}]$ is a codimension $1$ hyperplane. In $\Cone(Y,\varphi(x_n),c_n,\omega)$, the codimension $1$ hyperplane $[P_n']$ is contained in the intersection $[F_{1,n}'\cap F_{2,n}']$. Moreover  there are $z_n'\in \Nbhd_{d_n}(P_n')$ that have $d(z_n', P_n')=d_n$ and $d(z_n', F_{i,n}')\le D$. Therefore the point $[z_n']$ is contained in the intersection $[F_{1,n}']\cap [F_{2,n}']$, and this point is also distinct from $[L_n']$ since $d([z_n],[P_n'])=1$. Hence $[F_{1,n}]\cap [F_{2,n}]$ contains a strip with positive width containing the hyperplane $[P'_n]$. However this strip is the image of hyperplane $[P_n]$ under an isometry (up to a rescaling factor) $[\varphi]$. This is a contradiction.
\end{proof}

Recall that $\mathcal{G}$ is a full measure subset of $G$ such that for any $F\in\mathcal{F}$, there is $g\in \mathcal{G}$ such that $F=\pi(gA)$.

\begin{lemma}[Fubini's theorem]\label{Fubini}
Let $G$ be a group. $H$ be a subgroup and E a full measure subset of $G$. Then for a.e. $g\in E$, we have that $gh\in E$ for a.e. $h\in H$.
\end{lemma}

\begin{proof}See \cite[Lemma 5.1.1]{Dru}.
\end{proof}

By Fubini, for a.e. $g\in \mathcal{G}$, then $gk\in\mathcal{G}$ for a.e. $k\in K_\alpha$, for some $\alpha\in\Xi$. This is equivalent with saying that for almost every hyperplane of the form $P=\pi(gA_\alpha)$, almost every flat containing $P$ is in the family $\mathcal{F}$. Let $M$ be an arbitrary face in the building $\partial X$, denote $\Star(M)\subset \overline{K}$ be all the Weyl chambers containing $M$. If $M$ is a face in $P(\infty)$ where $P$ is the above hyperplane, then we see that $\Star(M)\cap\overline{\Omega}$ has full measure in $\Star(M)$. More precise, $\Star(M)\cap\overline{\Omega}$ contains a full measure subset of $\Star(M)$, that full measure subset consists of chambers that are in apartments bounding flats in $\mathcal{F}$ that contains the hyperplanes $P$. By Fubini again, almost every faces in $\partial X$ are faces with described properties of $M$.

For such a face $M$ as above, let $P$ be the hyperplane such that almost every flat containing $P$ is in $\mathcal{F}$, and $M\subset P(\infty)$. Assume that $P=\pi(gA_\alpha)$, then flat $\pi(gA)$ is sub-$\theta_\delta$-diverging w.r.t.~~ a large portion of points in $P$. Therefore we can assume that $g\in \Omega_\delta$. By proposition \ref{hyperprop}, there is a hyperplane $P'$ associated to $\varphi(P)$. For each Weyl chamber $E\in \Star(M)\cap\overline{\Omega}$, among $\varphi_1(E), \cdots, \varphi_l(E)$ there is (at least) one Weyl chamber adjacent to $P'(\infty)$. Without loss of generality, assume that is $\varphi_1(W)$.

\begin{proposition}\label{unicts}
The map $\varphi_1: \Star(M)\cap \overline{\Omega} \to \overline{K'}$ is uniformly continuous on a full measure subset. Moreover the extension map to $\Star(M)$ is injective.
\end{proposition}

\begin{proof}
The full measure subset $U$ of $\Star(M)$ where we are proving continuity is the subset consisting of chambers in apartments at infinity of  flats in $\mathcal{F}$ containing $P$.

For $W_1(\infty),W_2(\infty)\in U$ are chambers at infinity of $W_1,W_2$ vertex at $x$, we let $F_1,F_2$ be two flats in $\mathcal{F}$ containing hyperplane $P$ and two chambers $W_1, W_2$.

By the way we choose $P$ and $U$, $F_1$ and $F_2$ are sub-$\theta_\delta$-diverging w.r.t.~ a large portion of points in $P$. Let $x_1\in P$ be such a point, and let $W_3,W_4$ be Weyl chambers at $x_1$ that have $W_3(\infty)=W_1(\infty)$, $W_4(\infty)=W_2(\infty)$. Then for any $R>0$, $d_{Hau}(W_1\cap B(x,R),W_2\cap B(x,R))=d_{Hau}(W_3\cap B(x_1,R), W_4\cap B(x_1,R))$, and we have that $F_1,F_2$ are sub-$\theta_\delta$-diverging w.r.t.~ $x_1$.

Let the hyperplane $P'$ and flats $F_1',F_2'$ be the hyperplane and flats associated to $\varphi(P)$, and $\varphi(F_1), \varphi(F_2)$. By proposition \ref{hyperprop} $P'$ is determined up to some fixed finite Hausdorff neighborhood, hence projection of a point on $P'$ is also well-defined up to a finite distance. Let $y_1$ be the projection of $\varphi(x_1)$ to $P'$, and let $W_3'\subset F_1'$, $W_4'\subset F_2'$ be chamber vertex at $y_1$ such that
\[\varphi_1(W_1(\infty))=W_3'(\infty),\]
\[\varphi_1(W_2(\infty))=W_4'(\infty).\]

We also denote $y$ the projection of $\varphi(x)$ on $P'$, and $W_1', W_2'$ be chambers vertex $y$ with $W_1'(\infty)=W_3'(\infty)$, $W_2'(\infty)=W_4'(\infty)$. Note that $y$ is fixed (up to a finite distance), independent of chambers in $\Star(M)$.

Consider $\partial X$ with non-thick building structure induced from $\partial Y$, by corollary \ref{cham2cham}, there are non-thick chambers $V_1\subset W_3, V_2\subset W_4$ vertex at $x_1$ such that $\varphi(V_1)$, $\varphi(V_2)$ are asymptotic to $W_3', W_4'$ respectively.

Note that distance on Furstenberg boundary is biLipschitz equivalent with the visual metric at some base point. In other words, for any $\delta_0>0$, and any $R>0$, there exist $ \delta_1,\delta_2>0$ such that if $d_{\overline{K}}(W_1(\infty),W_2(\infty))<\delta_1$ (resp. $>\delta_2$) then $d_{Hau}(W_1\cap B(x,R), W_2\cap B(x,R)) < \delta_0 R$ (resp. $ >\delta_0 R$), where $W_1,W_2$ are Weyl chambers vertex at $x$ that have $W_1(\infty), W_2(\infty) \in \Star(M)$. Because $V_1, V_2$ are non-thick chambers, adjacent to $P$, there are $\delta_3, \delta_4$ such that if $d_{\overline{K}}(W_1(\infty),W_2(\infty))<\delta_3$ (resp. $>\delta_4$) then $d_{Hau}(V_1\cap B(x_1,R), V_2\cap B(x_1,R)) < \delta_0 R$ (resp. $ >\delta_0 R$).

By the estimate (\ref{refinedistance}),
\begin{align*}
d_{Hau}(\varphi(V_1)\cap B(y_1, L^{-1}(1-\beta(R))R-C),\varphi(V_2)\cap B(y_1, L^{-1}(1-\beta(R))R-C))\\< L(\delta_0 R+\beta(R)R)+C.\end{align*}
By corollary \ref{cham2cham}, if $L^{-1}(1-\beta(R))-C>R_\epsilon$, then
\begin{align*}
d_{Hau}(W_3'\cap B(y_1, L^{-1}(1-\beta(R))R-C),W_4'\cap B(y_1, L^{-1}(1-\beta(R))R-C))\\< L(\delta_0 R+\beta(R)R)+C +\epsilon (L^{-1}(1-\beta(R))-C).
\end{align*}
Choosing $R$ large enough, we can rewrite:
\[d_{Hau}(W_3'\cap B(y_1, \frac{1}{2}L^{-1}R),W_4'\cap B(y_1, \frac{1}{2}L^{-1}R))<\delta_0' R,\]
for some $\delta_0'$ deduced from above inequality. Thus
\[d_{Hau}(W_1'\cap B(y, \frac{1}{2}L^{-1}R),W_2'\cap B(y, \frac{1}{2}L^{-1}R))<\delta_0' R.\]
Again, using the equivalence of $d_{\overline{K'}}$ and the visual metric at $y$, there exists $\delta_2'$ such that if $d_{\overline{K'}}(V_1'(\infty),V_2'(\infty))>\delta_2'$ then $d_{Hau}(V_1'\cap B(y, \frac{1}{2}L^{-1}R),V_2'\cap B(y, \frac{1}{2}L^{-1}R))>\\ \delta_0' R$, where $V_1',V_2'$ are chambers in $Y$ vertex at $y$. Therefore, we conclude that \\$d_{\overline{K'}}(W_1'(\infty),W_2'(\infty))<\delta_2'$. This is equivalent with saying that for any $\delta_2'>0$ we can find $\delta_3>0$ such that for any pair of chambers in the full measure subset $U$ at $d_K$-distance at most $\delta_3$ then the image chambers under $\varphi_1$ are at $d_{\overline{K'}}$-distance $\delta_2'$. Hence we get the continuity of $\varphi_1$ on a full measure subset of $\Star(M)$.

To prove the injectivity of the extension map we repeat above argument for lower bound estimate and get
\begin{align*}
d_{Hau}(W_1'\cap B(y, 2R),W_2'\cap B(y, 2LR)&=
d_{Hau}(W_3'\cap B(y_1, 2R),W_4'\cap B(y_1, 2LR))\\&> L^{-1}(\delta_0 R-\beta(R)R)-C - 2\epsilon R =\delta_0'' R.
\end{align*}
Arguing as before we get for an $\delta_1'>0$, there is $\delta_4$ such that if $d_{\overline{K'}}(\varphi(W_1(\infty)),\varphi_1(W_2(\infty)))\\ <\delta_1'$ then $d_{\overline{K}}(W_1(\infty),W_2(\infty))<\delta_4$ for any $W_1(\infty),W_2(\infty)\in U$. This implies the injectivity of the extension map.
\end{proof}

\begin{corollary}
$\overline{\varphi}: \Star(M)\cap \overline{\Omega} \to \prod\limits_{i=1}^l \overline{K'}$ is uniformly continuous on a full measure subset.
\end{corollary}
\begin{proof}
Let $P$ be the hyperplane as above. Since $d_{\overline{K}}$ is biLipschitz with the visual metric at a point, we can assume that $d_{\overline{K}}$ is just the visual metric at a point in hyperplane $P$.

By the Proposition \ref{unicts}, there exists $ \varphi_i: \Star(M)\cap \overline{\Omega}\to\overline{K'}$ is uniformly continuous on a full measure subset $U\subset \Star(M)\cap \overline{\Omega}$, and chambers $\varphi_i(W)$ are adjacent with $P'(\infty)$ for all $W\in \Star(M)\cap\overline{\Omega}$.

Let $M^{op}$ be the opposite face to $M$ in $P(\infty)$. Consider $\overline{\varphi}:\Star(M^{op})\cap\overline{\Omega}\to \overline{K'}$. Because of the consistency when we define $\overline{\varphi}$, we have that $\varphi_i(W^{op})$ is opposite with $\varphi_i(W)$ and is adjacent to $P'(\infty)$ for any pair of opposite chambers $(W,W^{op})\in\ (\overline{\Omega}\cap \Star(M))\times (\overline{\Omega}\cap \Star(M^{op}))$. Let $U^{op}\subset \overline{\Omega}\cap \Star(M^{op})$ be the full measure subset that $\varphi_i$ is uniformly continuous on.

For every $ \epsilon >0$, there exists $ \delta_0$ such that for $W_1, W_2\in U$, if $d_{\overline{K}}(W_1,W_2)<\delta_0$ then \\$d_{\overline{K'}}(\varphi_i(W_1), \varphi_i(W_2))<\epsilon$, and for $W_3, W_4\in U^{op}$ with $d_{\overline{K}}(W_3,W_4)<\delta_0$ then $d_{\overline{K'}}(\varphi_i(W_3), \varphi_i(W_4))<\epsilon$.

Let $F_1, F_2$ be flats containing $P$ such that $W_1\subset F_1(\infty)$, $W_2\subset F_2(\infty)$. Let $W_1^{op}$, $W_2^{op}$ be opposite chambers with $W_1, W_2$ in $F_1(\infty)$ and $F_2(\infty)$. Because of what we assume on $d_{\overline{K}}$, if $d_{\overline{K}}(W_1,W_2)<\delta_0$ then $d_{\overline{K}}(W_1^{op}, W_2^{op})<\delta_0$.

Let $F_1', F_2'$ be flats in $Y$ associated with $\varphi(F_1), \varphi(F_2)$. The apartments $F_1'(\infty),\\ F_2'(\infty)$ have pairs of opposite chambers $(\varphi_i(W_1),\varphi_i(W_1^{op}))$ and $(\varphi_i(W_2),\varphi_i(W_2^{op}))$ are $\epsilon$-close. Hence apartments $F_1'(\infty), F_2'(\infty)$ are $\epsilon '$-close in Hausdorff metric, where $\epsilon '$ depends on $\epsilon$ and hyperplane $P'$. Therefore $\varphi_j$ is uniform continuous on $\Star(M)\cap \overline{\Omega}$ for all $j$. Note that all $\varphi_j$ are injective for all $j$ as well. This is because $F_1'(\infty), F_2'(\infty)$ share a common wall $P'(\infty)$ and $\varphi_i$ is injective.

Therefore $\overline{\varphi}$ is uniformly continuous on $U\subset \Star(M)\cap \overline{\Omega}$, and the extension of the map is also injective.
\end{proof}

\section{Regularity of boundary map}
\label{section:secondcontinuity}

The goal of this section is proving two following theorems:
\begin{theorem}\label{buildingmap}
There is a building monomorphism $\chi: \partial X\to \partial Y$ that agrees with $\overline{\varphi}$ a.e. on the set of chambers.
\end{theorem}

\begin{theorem}\label{buildinghomeo}
$\chi$ is continuous in the cone topology.
\end{theorem}

\begin{corollary}
$\chi(X)$ is a sub-building of $\partial Y$.
\end{corollary}

\begin{proof}
This is obvious from theorem \ref{buildingmap} and theorem \ref{buildinghomeo}.
\end{proof}

We start with some terminologies and definitions.
We know that the Coxeter group for $\partial X$ is a subgroup of the Coxteter group for $\partial Y$. Therefore, from now on when we say subCoxeter structure, we mean the structure on each apartment in $\partial Y$ with the Coxeter group is the one of $\partial X$.

\begin{definition}\emph{(subCoxeter admissible)}

A union of chambers/faces in $\partial Y$ is called a subCoxeter admissible (or admissible in short if there is no confusion) chamber/face if the union is contained entirely in in some apartment $\Sigma\subset \partial Y$ and there is an isometry from the modeled apartment for $\partial X$ into $\Sigma$ such that the union is exactly an image of a chamber/face.
\end{definition}

For example, if $\partial Y$ is a $B_n$ building, and the subCoxeter we consider is of type $D_n$, then any subCoexter admissible chamber is a union of two adjacent chambers having a common face of certain type. The other example is in our situation, image of a chamber in $\overline{K}\cap\overline{\Omega}$ under $\overline{\varphi}$ is a subCoxeter admissible chamber.

A pair of admissible chambers/faces are said to be adjacent if their intersection is a codimension 1 admissible face, and are said to be opposite if they are contained in an apartment and opposite in that apartment.

Let $L'$ be an admissible face. Denote $\Star(L')$ for the set of admissible chambers containing $L'$ as a face. There is a natural topology on $\Star(L')$, coming from the Hausdorff topology on $\partial Y$. This means, a sequence of admissible chambers $(C_n)$ in $\Star(L')$ is said to converge if they converge in Hausdorff topology. Therefore, for a face $L\subset \partial X$ and an admissible face $L'\subset\partial Y$, it makes sense to say a map $\tau:\Star(L)\to \Star(L')$ is continuous, injective, and adjacency preserving. If $M$ is a face of $D\in \Star(L)$, abusing notations, we use $\tau(M)$ as the admissible face of $\tau(D)$ corresponding to $M$ in the obvious way.

For $D, E$ subsets of $\partial X$ or $\partial Y$, we denote $CHull(D,E)$ for the combinatorial convex hull of $D$ and $E$ in $\partial X$ or $\partial Y$.

\begin{definition}\emph{(Coherence)}

Let $L, L^{op}$ be opposite faces in $\partial X$, $L', L'^{op}$ be subCoxeter admissible opposite faces in $\partial Y$. Two continuous adjacency preserving maps $\tau: \Star(L)\to \Star(L')$ and $\tau^{op}:\Star(L^{op})\to \Star(L'^{op})$ are said to be coherent if for any pair of chambers $D\in \Star(L)$, $E\in \Star (L^{op})$ such that $CHull(D,E)$ is a half apartment then $Chull (\tau(D), \tau^{op}(E))$ is also a half apartment.
\end{definition}

In order to prove theorem \ref{buildingmap}, we need a few lemmas.

\begin{lemma}\label{conlemma}
Given a half sphere $HA$ in the apartment model for a building, and $f_n:HA\to\partial Y$ be a sequence of isometries. Let $L$ and $L^{op}$ be opposite faces in $HA$, and let $L_n=f_n(L)$, $L_n^{op}=f_n(L^{op})$. Assume $L_n$ converge to a face $f(L)$, $L_n^{op}$ converge to a face $f({L}^{op})$ (in cone topology) where $f(L)$ and $f({L}^{op})$ are opposite. Then the restriction of $f_n$ on boundary converge to a isometry from $\partial HA$ to a wall containing $f(L)$ and $f({L}^{op})$. Furthermore, if there is an interior point $x$ such that $f_n(x)$ converge to a point $f(x)$, and $f(x)$ is not opposite with any $f(\zeta)$ for $\zeta\in\partial HA$ then $f_n$ converge to an isometry $f$. In particular, if $f(x)$ is an interior point of a chamber then $f_n$ converge to an isometry $f$.
\end{lemma}

\begin{proof}
The case rank 2 is obvious since a wall consists of exactly two opposite faces. So we only consider the case rank is higher than 2. Fix a based point $x_0$, let $d_{x_0}$ be the visual metric at $x_0$. The cone topology is equivalent with the topology induced from visual metric $d_{x_0}$.

Now let $z\in\partial HA$, there are two points $\xi\in L$ and $\zeta\in L^{op}$ such that $d_T(\xi, \zeta)<\pi$ and $z$ is on the geodesic connecting $\xi$ and $\zeta$. Let $z_n=f_n(z)$, $\xi_n=f_n(\xi)$, $\zeta_n=f_n(\zeta)$. For any $\epsilon>0$, since $(\zeta_n)$ and $(\xi_n)$ are Cauchy, we have that $d_{x_0}(\zeta_n,\zeta_m)<\epsilon$, $d_{x_0}(\xi_n,\xi_m)<\epsilon$ for $n,m$ large enough. Thus, for some fixed $\lambda>0$, there exists $R(\epsilon)>0$ such that $d(\overrightarrow{x_0\xi_n}(R), \overrightarrow{x_0\xi_m}(R))<\lambda$, and  $d(\overrightarrow{x_0\zeta_n}(R), \overrightarrow{x_0\zeta_m}(R))<\lambda$ for $n,m$ large enough. Here $\overrightarrow{x_0\zeta_n}(R)$ denote the point at distance $R$ from $x_0$ on the ray $(\overrightarrow{x_0\zeta_n})$. Because of the convexity of distance in CAT(0) space we have that $d(\overrightarrow{x_0z_n}(R\cos(\frac{d_T(\xi_n,\zeta_n)}{2})),\overrightarrow{x_0z_m}(R\cos(\frac{d_T(\xi_n,\zeta_n)}{2})))<\lambda$. Note that $d_T(\xi_n,\zeta_n)=d_T(\xi,\zeta)\\ <\pi$ constant, and $R(\epsilon)\to \infty$ as $\epsilon\to 0$. Hence, $(z_n)$ is a Cauchy sequence, thus converge to some point, denote $f(z)$. We easily see that $d(f(z),f(\xi))=d(f_n(z),f_n(\xi))$, $d(f(z),f(\zeta))=d(f_n(z),f_n(\zeta))$, so $f(z)$ is on the geodesic connecting $f(\xi)$ and $f(\zeta)$. Note that convex hull of a pair of opposite faces is a wall. Therefore, restriction of $f_n$ on boundary of $HA$ converge to an isometry on $\partial HA$.

Furthermore, assume $x$ is an interior point of $HA$ and $f_n(x)$ converge to $f(x)$. Then for any $\zeta\in\partial HA$, by assumption we have $d_T(f(x),f(\zeta))<\pi$. By the argument in the previous paragraph $f_n(z)$ converge to $f(z)$ for any $z$ in geodesic segment connecting $\zeta$ and $x$. And thus $f_n(x)$ converge to a point $f(x)$ for any $x\in HA$. That f is an isometry follows easily.
\end{proof}

For any sequence of half apartments having a pair of opposite faces converge to a pair of opposite faces, we treat them as a sequence of isometries from a fixed half apartment. And by position, we mean the images under the isometries at the same point.

\begin{lemma}\label{starextension}
Let $L, L^{op}$ be two opposite faces in $\partial X$, $L', L'^{op}$ admissible faces in $\partial Y$. Assume that there is an adjacency preserving map $\tau: \Star(L)\to \Star(L')$ that is continuous and for a.e. $F\in \Star(L)$, the convex hull $CHull(\tau(F), L'^{op})$ is a half apartment. Then for any $F\in \Star(L)$, the convex hull $CHull(\tau(F), L'^{op})$ is a half apartment.
\end{lemma}

\begin{proof}
Let $(F_n)$ be a sequence of chambers converging to $F$ in $\Star(L)$, and have $CHull(\tau(F_n),L'^{op})$ are half apartments. Those half apartments have common boundary, that is the convex hull of $L'$ and $L'^{op}$. So we get a sequence half apartment with common boundary, and the sequence of certain cells converging. By lemma \ref{conlemma} the sequence of half apartments converges to a half apartment. Hence $CHull(\tau(F),L'^{op})$ is a half aparment.
\end{proof}

\begin{lemma}\label{starchambermaplemma}
Let $L_n$ be a sequence of faces converging to a face $L$, and $L^{op}$ is face that is opposite with all $L_n$ and $L$. Assume that we have $\tau_n:\Star(L_n)\to \Star(L_n')$, where $(L_n')$ is a sequences of admissible faces converging to an admissible face $L'$. Assume there is $L'^{op}$ an opposite admissible face with all $L_n'$ and $L'$ and there is continuous adjacency preserving map $\tau^{op}: \Star(L^{op})\to \Star(L'^{op})$ such that all pair of maps $(\tau_n, \tau^{op})$ are coherent. Then there is a unique continuous adjacency preserving map $\tau: \Star(L)\to \Star(L')$ such that $\tau$ and $\tau^{op}$ are coherent.
\end{lemma}

\begin{proof}
We show that if $(D_n)$ is a sequence of chambers converging to the chamber $D$ where $D_n\in \Star(L_n)$ and $D\in \Star(L)$ then $\tau_n(D_n)$ converge to an admissible chamber adjacent to $L$, and we set the limit to be $\tau(D)$. Let $HA_n=CHull(L'^{op},\tau_n(D_n))$. As $\tau_n$ and $\tau^{op}$ are coherent, those sets are half apartments. First we have that their boundaries converge to a wall, that is the convex hull of $L'$ and $L'^{op}$. In the half apartment $CHull(D_n, L^{op})$, let $E_n$ be the chamber adjacent with $L^{op}$. Since $\tau_n$ and $\tau^{op}$ are coherent, $HA_n=CHull(\tau_n(D_n),\tau^{op}(E_n))$. The sequence of half apartments $CHull(D_n,E_n)$ have their boundaries converge, and interior chambers $D_n$ converge as well. By Lemma \ref{conlemma} the sequence of half apartments converge, in particular $E_n$ also converge to some limit chamber, say $E$. Clearly, $CHull (D,E)$ is a half apartment. Now $HA_n$ have boundaries converging to a wall and interior chambers $\tau^{op}(E_n)$ converge  to $\tau^{op}(E)$ due to continuity of $\tau^{op}$. Therefore the half apartments, and hence $\tau_n(D_n)$, converge as well. We set $\tau (D)$ be the limit. The coherence property of $\tau$ and $\tau^{op}$ follows immediately from from the definition of $\tau$. Continuity and uniqueness follow from coherence.
\end{proof}

\begin{proof} [Proof of theorem \ref{buildingmap}]
We already know that $\overline{\varphi}$ can be defined on chambers of almost every apartment, and $\overline{\varphi}$ sends those apartments to apartments. Let $\mathcal{A}$ be the family of apartments that bound flats in family $\mathcal{F}$. Then for almost every faces $L$ in $\partial X$, the subset of $\Star(L)$ consists of chambers $E$ such that $E$ is a chamber of some apartment in the family $\mathcal{A}$, has full measure in $\Star(L)$. We say an apartment $A\in\mathcal{A}$ is \textbf{good}, if for any faces $L\subset A$ then $\Star(L)$ has a full measure subset consists of chambers that belong to apartment in family $\mathcal{A}$. Then the family of good apartment is still of full measure in the set of apartments. Note that from previous section, if $L$ is a face of a good apartment $A\in\mathcal{A}$ then $\overline{\varphi}$ is uniformly continuous on a full measure subset of $\Star(L)$, and this full measure subset contains all chambers in $A$ that have $L$ as a face.

Let $\mathcal{E}\subset G$ be of full measure such that $\forall g\in \mathcal{E}$, $\pi(gA)(\infty)$ is a good apartment. In building $\partial X$, recall that there are $q$ chambers $C_1, \dots, C_q$ in an apartment ($C_q$ is opposite with $C_1$). There are $(q-1)$ subsets $\Delta_1,\dots,\Delta_{q-1}$ of $\Xi$ such that $D^+_{\Delta_i}(\infty)$ is the convex hull of $C_1$ and $C_i$. Here recall that $D^+_{\Delta_i}$ is a union of Weyl sectors consists of vector in $A$ with positive value when evaluated by all roots in $\Delta_i$. Let $U_{\Delta_i}$ be the unipotent subgroup of $G$ corresponding with the set of roots $\Delta_i$. By lemma \ref{Fubini}, there is $\mathcal{E'}\subset \mathcal{E}$ of full measure such that for all $g\in \mathcal{E'}$, and for each $i$, a.e. $h\in U_{\Delta_i}$, the apartment $\pi(ghA)(\infty)$ is good. Note that $\{\pi(ghA):h\in U_{\Delta_i}\}$ is the family of all apartments containing (convex hull of) chambers $gC_1$ and $gC_i$. Fix such $g$, let $C=gC_1$ and $\Sigma=\pi(gA)(\infty)$. And let $\Sigma '$, $C'$ be the apartment and admissible chamber corresponding to $\Sigma$ and $C$ via $\overline{\varphi}$.

Set $\overline{\omega} =op_\Sigma\circ retr_{\Sigma,C}$, where $retr_{\Sigma,C}$ is the retraction of $\partial X$ onto $\Sigma$ centered at $C$, and $op_\Sigma$ is the map sending a chamber/face to the opposite one in $\Sigma$. Similarly, we set $\overline{\omega}'=op_{\Sigma'}\circ retr_{\Sigma',C'}$. This is, however, not always well defined on set of admissible chambers/faces, but we will only consider the map on whichever admissible chambers/faces it can be defined. We also denote $\alpha: \Sigma\to\Sigma'$ the isomorphic complex map that is restriction of $\overline{\varphi}$ on $\Sigma$.

For every face $L\subset \Sigma$, by assumption, there is an admissible face $L'\subset \Sigma'$, and a continuous adjacency preserving map $\varphi_L: \Star(L)\to \Star(L')$ that coincides a.e. with $\overline{\varphi}|_{\Star(L)\cap \overline{\Omega}}$. We let $\psi(L)=L'$. We also have that $\varphi_L$ is coherent with $\varphi_{\overline{\omega} (L)}$ for all faces $L\subset \Sigma$.

For general L, we prove by induction on the combinatorial distance from $C$ the following:

\textit{For every face $M\subset \partial X$ of distance at most $k+1$ from $C$, there is an admissible face $\psi (M)\subset \partial Y$, and for each face $L\subset\partial X$ of distance at most $k$ from $C$, there is a continuous adjacency preserving map $\varphi_L: \Star(L)\to \Star(\psi(L))$ with the following properties:
\begin{enumerate}
\item $\overline{\omega}'(\psi (M))=\alpha (\overline{\omega}(M))$.\label{1}
\item $\varphi_L$ and $\varphi_{\overline{\omega}(L)}$ are coherent.\label{2}
\item if $M\subset \partial X$ is a face in a chamber containing $L$ such that $dist(C,M)=\\dist(C,L)-1$ or $dist(C,M)=dist(C,L)$, then $\varphi_M$ and $\varphi_L$ coincide on the chamber containing both $L$ and $M$.\label{3}
\item If $L$ is a face in a good apartment containing $L$, $\overline{\omega}(L)$, and $C$ then $\varphi_L$ agrees with $\overline{\varphi}$ on a full measure subset of $\Star(L)$. If $N$ is a face in a good apartment containing $N$, $\overline{\omega}(M)$, and $C$ then $\psi(M)$ agrees with image of $M$ under $\overline{\varphi}$ restricted on that good apartment.\label{4}
\item The restriction of $\psi$ to the set of faces of distance $k+1$ from $C$ is continuous.\label{5}
\end{enumerate}
}

Indeed, when $k=0$ then $L$ is a face of $C$, and $\varphi_L$ and $\psi(L)$ already exist. Then (\ref{2}), (\ref{3}) and the first half of (\ref{4}) will be immediate. Denote $E_1(C)$ for set of chambers that are adjacent to $C$ by a codimension 1 face of $C$. If $M$ is a face at distance 1 from $C$, then there is $ D\in E_1(C)$ such that $M$ is a face of $D$. Assume that $L=C\cap D$, then $\psi(M)$ can be defined as $\varphi_L(M)$. We know (\ref{1}) is true as $\varphi_L(D)$ and $\varphi_{\overline{\omega}(L)}(\overline{\omega}(D))$ are opposite due to the coherence of $\varphi_L$ and $\varphi_{\overline{\omega}(L)}$. For (\ref{5}), if $M_n$ are faces at distance 1 from $C$ that converge to a face $M$, also at distance 1. Let $D,D_n\in E_1(C)$ such that $M$ and $ M_n$ is a face of $D$ and $D_n$ respectively, and assume $L_n=C\cap D_n$, $L=C\cap D$. As $D_n$ converge to $D$, $L_n$ also converge to $L$. Hence $L_n=L$ for large $n$. Hence $\varphi_{L_n}(D_n)$ converge to $\varphi_L(D)$ due to the continuity of $\varphi_L$. Therefore $\psi(M_n)$ converge to $\psi(M)$ as well.

Assume by induction that $\varphi_L$ exists with above properties up to faces $L$ have distance at most $k-1$ from $C$, and $\psi$ is defined for faces at distance up to $k$. Now let $L$ be a face of distance $k$. Let $D$ be a chamber containing $L$ and is on a combinatorial geodesic path from $C$ to $L$. Let $M$ be a face of $D$ that has distance $k$ to $C$. Note that $D$ and $\overline{\omega}(D)$ determine an apartment, and this apartment contains the convex hull of $\overline{\omega}(D)$ and $C$. By assumption, a.e. apartments containing the convex hull of $\overline{\omega}(D)$ and $C$ are good. Thus there is a sequence of good apartments containing the convex hull of $\overline {\omega}(D)$ and $C$, converging to the apartment containing $\overline{\omega}(D)$ and $D$. Hence, there exists a sequence of faces $L_n$ in that sequence of good apartments such that $L_n$ converge to $L$. Moreover for each $L_n$ there is a continuous adjacency preserving map $\tau_n: \Star(L_n)\to \Star(\psi(L_n))$ that agrees with $\overline{\varphi}$ almost everywhere on $\Star(L_n)$. $\tau_n$ sends star chamber of $L_n$ to star chamber of $\psi(L_n)$ is due the property (\ref{4}) on induction assumption of $\psi$. Since $\tau_n$ agrees with $\overline{\varphi}$ on a full measure set of $\Star(L_n)$, by lemma \ref{starextension}, $\tau_n$ and $\varphi_{\overline{\omega}(N)}$ are coherent. Because of the induction assumption (\ref{5}) on continuity on the set of faces at distance $k$ from $C$, we have that $\psi(L_n)$ converge to $\psi(L)$. $\psi(L)$ is opposite with $\psi(\overline{\omega}(L))$ because of (\ref{1}). By the lemma \ref{conlemma} there is a continuous adjacency preserving map $\varphi_L: \Star(L)\to \Star(\psi(L))$ that is coherent with $\varphi_{\overline{\omega}(L)}$. With different choices of sequences of good apartments and $L_n$, we still get an adjacency preserving map from star chamber of $L$ to star chamber of $\psi(L)$ due to the continuity of $\psi$ by (\ref{4}). And the map $\varphi_L$ is defined uniquely due to the coherence with $\varphi_{\overline{\omega}(L)}$.

Let $N$ be a face at distance $k+1$ from $C$. Assume that $N$ is a face of $E\in \Star(L)$ where $dist(L,C)=k$. Define $\psi(N)=\varphi_L(N)$. We have to prove $\psi(N)$ is well defined. The convex hull of $N$ and $C$ contains only one chamber having $N$ as a face, i.e. there is only one chamber, that is $E$, containing $N$ and such that $dist(E,C)=k+1$.  Therefore, if $M$ is a face contained in the combinatorial path from $C$ to $N$ and such that $dist(C,M)=k$ then $M$ is a face of $E$. To show $\varphi_L$ and $\varphi_M$ map the face $N$ into the same image, we show that $\varphi_L(E)=\varphi_M(E)$. Because of the way we choose apartment $\Sigma$, there is a sequence of good apartments containing $CHull(\overline{\omega}(E),C)$ that converges to the apartment containing $E$ and $\overline{\omega}(E)$. Pick the corresponding sequences of chambers $(E_n)$ and faces $(L_n)$ and $(M_n)$ where $L_n, M_n$ are faces of $E_n$ and such that $E_n\to E$, $L_n\to L$, $M_n\to M$ and $\overline{\omega}(E_n)=\overline{\omega}(E)$, $\overline{\omega}(L_n)=\overline{\omega}(L)$, $\overline{\omega}(M_n)=\overline{\omega}(M)$. On the image we know that $\psi(L_n)\to \psi(L)$ and $\psi(M_n)\to \psi(M)$ due to the continuity of $\psi$ on the set of faces at distance $k$ from $C$. Recall that $\varphi_L(E)$ and $\varphi_M(E)$ are defined as limits of $\tau_{L_n}(E_n)$ and $\tau_{M_n}(E_n)$ respectively, where $\tau_{L_n}$ and $\tau_{M_n}$ are coherent with $\varphi_{\overline{\omega}(L_n)}$, $\varphi_{\overline{\omega}(M_n)}$ and agrees with $\overline{\varphi}$ a.e. on $\Star(L_n)$, $\Star(M_n)$. Due to the coherence, $\tau_{L_n}(E_n)=\tau_{M_n}(E_n)$, implies that $\varphi_L(E)=\varphi_L(E)$. Thus $\varphi_L(E)=\varphi_M(E)$. Hence the image of the face $N$ is well-defined when we set $\psi(N)=\varphi_L(N)=\varphi_M(N)$.

We now verify properties (\ref{1}) - (\ref{5}) of $\varphi$ and $\psi$. Note that the coherence (\ref{2}) is immediate from the way we defined $\varphi_-$.
\begin{enumerate}
\item Let $N$ be a face at distance $k+1$ from $C$. Assume that $N$ is a face of $E\in \Star(L)$ where $dist(L,C)=k$. Because $\varphi_L$ is coherent with $\varphi_{\overline{\omega}(L)}$, we have that $\overline{\omega}'(\varphi_L(E))=\alpha(\overline{\omega}(E))$. Hence $\overline{\omega}'(\psi(N))=\alpha(\overline{\omega}(N))$.

\item $\varphi_L$ and $\varphi_{\overline{\omega}(L)}$ are coherent because of the way we defined $\varphi_L$. This coherence property will make $\varphi_L$ be uniquely defined.

\item Let $M\subset \partial X$ is a face in a chamber containing $L$ such that $dist(C,M)=\\ dist(C,L)-1$ or $dist(C,M)=dist(C,L)$. Let E be the chamber containing $L$ and $M$.We treat each case separately.

Case 1: $dist(C,M)=dist(C,L)-1$. Because of the induction assumption (\ref{1}), we have that $\varphi_M(L)=\psi(L)$. Consider the apartment containing $E$ and $\overline{\omega}(E)$. Since $\varphi_M$ and $\varphi_{\overline{\omega}(M)}$ are coherent, this apartment is mapped to an apartment in $\partial Y$ and $\varphi_M(E)$ is opposite with $\varphi_{\overline{\omega}(M)}(\overline{\omega}(E))$. This apartment also contains $\psi(L)$ and $\psi(\overline{\omega}(L))$, as they are admissible faces of $\varphi_M(E)$ and $\varphi_{\overline{\omega}(M)}(\overline{\omega}(E)$. Moreover, $\varphi_L$ and $\varphi_{\overline{\omega}(L)}$ are coherent, thus $\varphi_L(E)$ is the admissible chamber opposite to $\varphi_{\overline{\omega}(L)}(\overline{\omega}(E))$ in this image apartment. But $\varphi_{\overline{\omega}(L)}(\overline{\omega}(E))$ is the same as $\varphi_{\overline{\omega}(M)}(\overline{\omega}(E))$. Hence $\varphi_L(E)=\varphi_M(E)$.

Case 2: $dist(C,M)=dist(C,L)$. The same argument as when we showed $\psi$ is well-defined, can be applied in this case to conclude that $\varphi_L(E)=\varphi_M(E)$.

\item If $M$ is a face at distance $k+1$ from $C$ in a good apartment $A$ containing $C$ and $\overline{\omega}(M)$ . Let $L$ be a face at distance $k$ on the geodesic combinatorial path from $C$ to $M$. Because of the way we defined $\varphi_L$, it is obvious that $\varphi_L$ agrees with $\overline{\varphi}$ on a full measure subset of $\Star(L)$. Let $E$ be the chamber containing $L$ and $M$. By the observation we made when we defined good flat, the subset of $\Star(L)$ where $\overline{\varphi}$ is uniformly continuous on contains $E$. Therefore $\psi(M)=\varphi_L(M)$ agrees with image of $M$ by restriction of $\overline{\varphi}$ on the good apartment $A$.

\item Let $L_n$ be a sequence of faces at distance $k+1$ from $C$, and converge to a face $L$, also at distance $k+1$ from $C$. From some $k$ large enough, $\overline{\omega}(L_n)=\overline{\omega}(L)$, therefore, without loss of generality we assume $\overline{\omega}(L_n)=\overline{\omega}(L)$ for all $n$. By the lemma \ref{conlemma}, the sequence of half apartments containing $L_n$, $\overline{\omega}(L_n)$, and $C$ converge to the half apartment containing $L$, $\overline{\omega}(L)$, and $C$. Therefore the sequence of chambers $E_n$ converge to the chamber $E$, where $E_n$ and $E$ are chambers containing $L_n$ and $L$ such that $dist(E_n,C)=dist(E,C)=k+1$. Let $(M_n)$ be a sequence of faces of $E_n$ such that $dist(M_n,C)=k$ and $M_n$ converge to a face $M$ of $E$. Because of the convergence of the sequence of half apartments, we can assume that $\overline{\omega}(M_n)=\overline{\omega}(M)$, and $\overline{\omega}(E_n)=\overline{\omega}(E)$. By induction, $\psi(M_n)$ converge to $\psi(M)$, and $\varphi_{M_n}$, $\varphi_M$ are coherent with $\varphi_{\overline{\omega}(M)}$. By Lemma \ref{starchambermaplemma}, $\varphi_{M_n}$ converge to $\varphi_M$ due to the coherence of $\varphi_M$ and $\varphi_{\overline{\omega}(M)}$. In particular, $\varphi_{M_n}(E_n)$ converge to $\varphi_M(E)$. It follows that $\psi(L_n)$ converge to $\psi(L)$. Hence restriction of $\psi$ on the set of faces of distance $k+1$ from $C$ is continuous.
\end{enumerate}

So we have proven the existence of $\psi$ and $\varphi_-$ with above properties. This induces a map $\chi: \partial X \to \partial Y$, defined as following points of image $\chi(E)$ are points of admissible chamber $\varphi_L(E)$ for some face $L$ of the chamber $E$. $\chi$ is well-defined because of the property (\ref{3}). Note that $\chi$ is adjacence preserving and injective on each star chamber, hence $\chi$ is a building monomorphism into the image $Z\subset\partial Y$. The fact that $\chi$ agrees with $\overline{\varphi}$ a.e. follows from the (\ref{4}) and the way we picked the apartment $\Sigma$.
\end{proof}

\begin{proof}[Proof of theorem \ref{buildinghomeo}]
Since $\chi$ maps each chamber in $\partial X$ to an admissible chamber in $\partial Y$, in order to prove $\chi$ is continuous on cone topology, we prove that the map $\chi$ is continuous with respect to Furstenberg boundary topology. There exists $ m$, such that for almost every tuple of $m$ chambers, the Furstenberg boundary can be written as union of $m$ open subsets consisting of opposite chambers with the ones in the m-tuple. Thus, there is a tuple $(C_1,\dots,C_m)$ of chambers such that:
\begin{itemize}
\item $C_i$ satisfies all property as chamber $C$ we pick at the beginning of the proof of theorem 2, for all $i=1,\dots,m$.
\item $C_i$ and $C_j$ are pairwise opposite for $i\neq j$. And apartments containing each pair $C_i$, $C_j$ are all good.
\item The Furstenberg boundary can be written as union of $m$ open sets $\Omega_1,\dots, \Omega_m$, where $\Omega_i$ is the set of opposite chambers with $C_i$.
\end{itemize}
Therefore, we only need to prove $\chi$ is continuous on each open set $\Omega_i$.  Let $\Sigma_i$ be the good apartment containing $C$ and $C_i$. Set $\overline{\omega}_i =op_\Sigma\circ retr_{\Sigma_i,C_i}$. Proceed inductively on combinatorial distance $k$ from chamber $C_i$ as the proof of of Theorem \ref{buildingmap}, that the restriction $\psi$ on sphere radius $k$ of faces around $C_i$ is continuous. For $k=1$, this claim is true because of the continuity of $\varphi_M$ where $M$ is any face of $C_i$. Suppose that the claim is true up to distance $k$. Note that $\chi$ restrict to any apartment is an isomorphism, hence for any face $L$, $\varphi_L$ and $\varphi_{\overline{\omega}_i(L)}$ are coherent. Then we are in the same situation as proof of the (\ref{5}) of $\psi$ in the proof of theorem 2. Therefore by the same argument, we conclude that restriction of $\psi$ is continuous on each sphere of faces around $C_i$. Now assume that we have a sequence of chambers $(W_n)$ converging to a chamber $W$, and they are all in $\Omega_i$. Let $(L_n)$ be a sequence of faces of $(W_n)$ that converge to a face $L$ of $W$. For $n$ large enough then $\overline{\omega}_i(L_n)=\overline{\omega}_i(L)$, thus without loss of generality, we can assume this for all $n$. Let $E_n$ be the chamber adjacent to $C_i$ via the face $\overline{\omega}_i(L)$ and is in the apartment containing $C_i$ and $W_n$. Because $W_n$ converge to $W$, $E_n$ also converge to a chamber $E$ in the apartment containing $C_i$ and $W$, so that $C_i$ and $E$ are adjacent via the face $\overline{\omega}_i(L)$. The sequence of half apartment containing $\psi(L_n)$, $\psi(\overline{\omega}_i(L))$, and $\chi(E_n)=\varphi_{\overline{\omega}_i(L)}(E_n)$, have opposite faces converge to a pair of opposite faces and chambers converge to a chamber, thus by Lemma \ref{conlemma}, the sequence of half apartments converges. In particular $\chi(W_n)=\varphi_{L_n}(W_n)$ converge to a chamber. This chamber has to be $\chi(W)=\varphi_L(W)$ due to the coherence of $\varphi_L$ and $\varphi_{\overline{\omega}_i(L)}$. Therefore $\chi$ is continuous on $\Omega_i$. It follows that $\chi$ is continuous.
\end{proof}

\section{Rigidity}

From previous section, we have that $\partial Y$ contains a sub-building $\chi(\partial X)$ that is isomorphic with $\partial X$. By \cite[Theorem 3.1]{KL1}, there is a symmetric subspace isometric to $X$ of $Y$ that has the boundary is $\chi(\partial X)$. Therefore, $G$ is a subgroup of $G'$, and the map $\chi$ is given by the action of a element $g\in G<G'$ on boundary.
\begin{proposition}
There is $C$ such that for all $\gamma\in \Gamma$, $d(\varphi(\gamma),\pi(g\gamma))<C$.
\end{proposition}

\begin{proof}
By orthogonal flats, we mean a pair of flats that intersect at one point, and moreover the intersection of their $r$-tubular neighborhoods is contained in a ball of radius $\lambda r$ for some fixed $\lambda$ (see lemma 7.2 in \cite{Esk}). Let $\Omega_\delta '\subset \Omega_\delta$, as a subset of a fundamental domain for $\Gamma\backslash G$, consists of elements $h$ that have two orthogonal flats in $\mathcal{F}$ passing through. To be clear, this means there exist $ k_1,k_2\in K$ such that $hk_1,hk_2\in \Omega_\delta$. Then by ergodic and Fubini's theorem, there is $c$, depends only on the space $X$, such that $\mu(\Omega_\delta ')>1-c\delta$. Intersect with a compact subset of the fundamental domain, we get a set $\Omega_{\delta,C}\subset\Omega_\delta '$ such that $\mu(\Omega_{\delta,C})>1-2c\delta$, and diameter of $\Omega_{\delta,C}$ is smaller than $C$.

For any $\gamma\in \Gamma$, pick $u\in\Omega_{\delta,C}$, then $d(\pi(\gamma),\pi(\gamma u))<C$ and there are flats $F_1, F_2\in \mathcal{F}$ orthogonal at $\pi(\gamma u)$. There exists $F_1', F_2'\subset Y$ such that $[\varphi(F_1)]=[F_1']$, and $[\varphi(F_2)]=[F_2']$. Moreover, we know that $g$ agrees with $\overline{\varphi}$ on good apartments. Thus $F_1'(\infty)=gF_1(\infty)$, $F_2'(\infty)=gF_2(\infty)$, and it follows that $F_1'=gF_1$, $F_2'=gF_2$. Therefore, $F_1', F_2'$ are orthogonal at $\pi(g\gamma u)$. But $\varphi(\pi(\gamma u))$ is $D$-close to both flats $F_1', F_2'$ as $F_1,F_2$ are both sub-$\theta_\delta$-diverging w.r.t.~ $\pi(\gamma u)$. Hence, $d(\varphi(\gamma u), \pi(g\gamma u))<D'$, where $D'$ depends on $D$ and $\lambda$ we choose. Note that $\gamma u$ is $C$-close to the lattice $\Gamma$, implies $d(\varphi(\gamma),\varphi(\gamma u))$ is bounded by a universal amount too. Hence there is a constant $C$ such that $d(\varphi(\gamma),\pi(g\gamma))<C$ for all $\gamma\in \Gamma$.
\end{proof}

\begin{proof}[Proof of theorem \ref{mainthm}]
Since $\varphi: \Gamma\to \Lambda$, and $\varphi$ is uniformly close to the action of $g$ on $\Gamma$ by previous proposition, we get that $g\Gamma$ is contained in a finite neighborhood of $\Lambda$.

By \cite[Theorem 1.3]{Sha}, there is a closed subgroup $H<G'$ containing $\Gamma$, such that $\overline{p'( g\Gamma)}=p'(gH)$ here $p':G'\to \Lambda\backslash G'$ is the projection, and the bar stands for closure of a set in $\Lambda\backslash G'$. In other words, $\overline{\Lambda g\Gamma}=\Lambda gH$. This implies that  $\overline{\Lambda g\Gamma g^{-1}}=\Lambda gHg^{-1}$, and it follows that $\Gamma$ normalizes $H^{0}$. Since $\Gamma$ is a lattice in $G$, by Borel density, $G$ normalizes $H^0$ as well. This implies that $H^0\cap G = G$ or $H^0\cap G =\{1\}$.
On the other hand as $g\Gamma$ is in a finite neighborhood of $\Lambda$, $\overline{p'(g\Gamma)}$ is compact. By our assumption, $H$ is discrete, and the orbit $\Lambda g\Gamma=\Gamma gH$ consists of finitely many points. It follows that orbit $\Lambda g\Gamma g^{-1}=\Lambda gHg^{-1}$ also consists of finitely many points. Therefore there exists $ \Gamma '<\Gamma$ of finite index such that $g\Gamma' g^{-1} < \Lambda$. Note that the map $g: \gamma\mapsto g\gamma$ is uniformly close to the homomorphism $Ad_g: \gamma\mapsto g\gamma g^{-1}$. Hence, the quasi-isometric  embedding map $\varphi$ is uniformly close to a virtually monomorphism $\Gamma \to \Lambda$.
\end{proof}

\noindent{\bf Remark:}
Without our assumption on the non-existence of continuous group with compact orbit, we could derive the following:

In the case  $H^0\cap G = G$. In this case $H$ is a subgroup containing $G$ and intersects $\Lambda$ in a uniform lattice in $H$. Therefore the quasi-isometric embedding is uniformly close to a projection of a discrete subgroup to a uniform lattice in $H$ composed with the inclusion of that lattice into $\Lambda$.

In the case  $H^0\cap G =\{1\}$. The case $H$ is discrete has been treated above when we have the assumption. If $H$ is continuous, then $H^0$ is an algebraic subgroup that contains a finite index subgroup of $\Gamma$, hence contains $G$ as well. This contradicts with $H^0\cap G =\{1\}$.

\begin{appendix}

\section{by Skip Garibaldi, D.~B.~McReynolds,
\\ Nicholas Miller, and Dave Witte Morris}

This appendix constructs examples where Condition~(2) of Theorem \ref{mainthm} holds, and also identifies a few situations in which the condition is impossible to satisfy.
The main results are Examples~\ref{SLnRinSLnCBad}, \ref{SO6bad}, \ref{SO2rBad}, \ref{SOequalBad} and Propositions~\ref{NoSLZi}, \ref{SOinSO}.

\begin{lem} \label{DivAlgForAEg}
If $F/\QQ$ is an imaginary quadratic extension, then for every $n \ge 2$, there is a central division algebra $D$ of degree $2n$ over~$\QQ$ such that $D$ splits over~$\RR$ and $D \otimes_{\QQ} F$ is not a division algebra.
\end{lem}

\begin{proof}
Let $D_2$ be any quaternion division algebra over~$\QQ$ that splits over both $\RR$ and~$F$. For example, if $F = \QQ[\sqrt{a}]$, we can take $D_2 = \left( \frac{\textstyle a,b}{\textstyle \QQ} \right)$ for any positive rational number $b$ that is not a norm in~$F$. Next, let $D_n$ be any central division algebra of degree~$n$ over~$\QQ$, such that $D_2 \otimes_{\QQ} D_n$ is a division algebra but $D_n$ splits over~$\RR$. If $n$ is odd, then $D_n$ can be any central division algebra of degree~$n$ over~$\QQ$. For even $n$, there are local restrictions that can be arranged with some mild care. The sought after algebra $D$ can be taken to be $D = D_2 \otimes_{\QQ} D_n$.
\end{proof}

\begin{eg} \label{SLnRinSLnCBad}
For $n \ge 2$, there is a noncocompact lattice in $\SL_{2n}(\CC)$ such that $\Lambda \cap \SL_{2n}(\RR)$ is a cocompact lattice.
\end{eg}

\begin{proof}
Let $D$ and~$F$ be as in Lemma~\ref{DivAlgForAEg}. Since $D \otimes_{\QQ} F$ is a central algebra over~$F$ and $F$ is imaginary, we know that $\SL_1(D \otimes_{\QQ} F)$ is a $\QQ$-form of $\SL_{2n}(\CC)$. Also, it is isotropic because $D \otimes_{\QQ} F$ is not a division algebra. Moreover, as $D$ is a central division algebra over~$\QQ$ and splits over~$\RR$, we know that $\SL_1(D)$ is an anisotropic $\QQ$-form of $\SL_{2n}(\RR)$. By construction, $\SL_1(D)$ is contained in $\SL_1(D \otimes_{\QQ} F)$. Passing to the $\ZZ$-points of these groups provides the desired lattices.
\end{proof}

However, the following result implies that the lattice in Example~\ref{SLnRinSLnCBad} cannot be conjugate to $\SL_n \bigl( \ZZ[i] \bigr)$.

\begin{prop} \label{NoSLZi}
If $n \ge 3$ and $\ints$ is the ring of integers of an imaginary quadratic extension~$F/\QQ$, then there does not exist a closed subgroup~$G$ of $\SL_n(\CC)$ such that $G$ is isogenous to $\SL_n(\RR)$ and $G \cap \SL_n(\ints)$ is a cocompact lattice in~$G$.
\end{prop}

\begin{proof}
Let $\Lie G \subseteq \sl_n(\CC)$ be the Lie algebra of~$G$. If $G \cap \SL_n(\ints)$ is a cocompact lattice in~$G$, then $\Lie G_{\QQ} = \Lie G \cap \sl_n(F)$ is an anisotropic $\QQ$-form of~$\Lie G$. Consequently $\Lie G_{\QQ} \otimes_{\QQ} F$ is an $F$-Lie subalgebra of $\sl_n(F)$ and since it is of type~$A_{n-1}$, it cannot be a proper subalgebra. Therefore $\Lie G_{\QQ} \otimes_{\QQ} F = \sl_n(F)$ and $\Lie G_{\QQ}$ splits over~$F$.

Since $\Lie G_{\QQ}$ splits over both $\RR$ and~$F$, so it is inner over both of these fields. Therefore, it is inner over their intersection, which is~$\QQ$; that is, $\Lie G_{\QQ}$ is an inner~$\QQ$-form.  From the classification of (anisotropic, inner) $\QQ$-forms of $\SL_n$, this implies $\Lie G_{\QQ} = \sl_1(D)$, for some central division algebra~$D$ over~$\QQ$. As $\Lie G_{\QQ}$ splits over the quadratic extension~$F$, we see that $D$ must be a quaternion. Consequently, $n = 2$, which contradicts the assumption that $n \ge 3$.
\end{proof}

We now turn to the task of giving some restrictions on the possible type of $H$, if such an $H$ exists.
That is accomplished by Corollary \ref{OnlyA}.
G.~Harder proved the following theorem under the assumption that the group is not of type~$E_8$, but J.~Tits \cite[p.~669]{Tits-StronglyInner} pointed out that this assumption is no longer needed, because V.~Chernousov subsequently proved the Hasse principle (Harder's Satz 4.3.1) for~$E_8$.

\begin{thm}[Harder {\cite[Satz~4.3.3]{Harder-BerichtNeuere}}] \label{IsotropicIffLocally}
If a vertex of the Tits index of a simple $\QQ$-group is circled at every place, then it is circled in the Tits index over~$\QQ$.
\end{thm}

\begin{cor} \label{OnlyA}
Assume $H$ is an almost simple, closed, noncompact subgroup of $\GL_{n}({\RR})$, for some~$n$. If $\Rrank {H} \ge 2$ and $H \cap \GL_n(\ZZ)$ is a cocompact lattice in~$H$, then $H$ is either of type~$A_n$, for some~$n$, or of type ${}^1\!E_{6,2}^{28}$ \textup(over~$\RR$\textup).
\end{cor}

\begin{proof}
Since $H \cap \GL_n(\ZZ)$ is a lattice in~$H$, the Borel Density Theorem implies that $H$ is (of finite index in) a $\QQ$-subgroup of $\GL_n(\RR)$. Furthermore, since this lattice is cocompact, we know that $H$ is anisotropic over~$\QQ$. This means that no vertex is circled in the Tits index of the $\QQ$-group~$H$. However, by inspection of the list of Tits indices in \cite[pp.~55--61]{Tits-Classification}, we see that for each type except ${}^{1,2}\!A_n$ and ${}^{1,2}\!E_6$, there is a vertex that is circled for all $\RR$-forms of rank $\ge 2$ and also all $p$-adic forms:
	\begin{itemize} \itemindent=2\parindent
	\item[$B_n$:] the leftmost vertex is circled.
	
	\item[$C_n$:] the 2nd vertex from the left is circled.
	
	\item[${}^{1,2,3,6}D_n$:] the 2nd vertex from the left is circled (in $D_4$, this is the central vertex).
	
	\item[$E_7$:] the rightmost vertex is circled.
	
	\item[$E_8$:] the leftmost vertex is circled.
	
	\item[$F_4$:] all vertices are circled.
	
	\item[$G_2$:] both vertices are circled.
	\end{itemize}
Furthermore, for ${}^{1,2}\!E_6$, the end of the short leg is circled in every $p$-adic Tits index, and is circled in every isotropic index over~$\RR$ except ${}^1\!E_{6,2}^{28}$. Therefore, Theorem~\ref{IsotropicIffLocally} implies that $H$ is of type $A_n$ or ${}^1\!E_{6,2}^{28}$.
\end{proof}

For the special case where $G$ is isogenous to $\SO(n,m)$ and $G'$ is isogenous to $\SO(n,m+\ell)$, we now give some numerical conditions that imply hypothesis~(2) of Theorem \ref{mainthm} is satisfied.

\begin{prop} \label{SOinSO}
Assume
	\begin{itemize}
	\item $G \le H \le G'$,
	\item $\Lambda$ is a noncocompact lattice in~$G'$, such that $\Lambda \cap H$ is cocompact in~$H$,
	\item $G$ is isogenous to $\SO(n,m)$,
	\item $G'$ is isogenous to $\SO(n,m+\ell)$,
	\item $H$ is \textup(closed and\/\textup) connected,
	and
	\item $2 \le n \le m \le m+\ell$.
	\end{itemize}
If $n + m \ge 7$, then $\ell \ge n + m$.
\end{prop}

\begin{proof}
We proceed via contradiction and assume that $\ell< n+m$. First, we show that $H$ is reductive. If not, then it is contained in a proper parabolic subgroup~$P$ of~$G'$ \cite[Prop.~3.1(ii)]{BorelTits}. Letting $P = MAN$ be a Langlands decomposition, we know that the $\RR$-split torus~$A$ is nontrivial. As $G$ is contained in (a conjugate of) $M$ and $MA \subset G'$, we see that $\rank_{\RR} G < \rank_{\RR} G'$, which contradicts the observation that $\rank_{\RR} G = n = \rank_{\RR} G'$. Hence, $H$ is reductive. In fact, since $\rank_{\RR} G = \rank_{\RR}G'$, it is clear that the center of~$H$ must be compact. So there is no harm in assuming it is trivial, which means that $H$ is semisimple. Assuming, as we may, that $\Lambda \cap H$ is irreducible, we know that $H$ is isotypic. As $\ell < n + m$, we see that $H$ is almost simple. Since $\Rrank G'=n\ge 2$, the Margulis Superrigidity Theorem implies that $\Lambda$ is arithmetic. As $\Lambda$ is not cocompact, we know $\Lambda$ is commensurable with the $\ZZ$-points of $G'$ for some $\QQ$-structure on $G'$.  In particular, $\Lambda \cap H$ is commensurable with the $\ZZ$-points of~$H$.  Hence, Corollary~\ref{OnlyA} implies that either $H_{\CC}$ is isogenous to $\SL_r(\CC)$ (or $\SL_r(\CC) \times \SL_r(\CC)$), or $H$ is of type~${}^1\!E_{6,2}^{28}$.

\setcounter{case}{0}
\begin{case}
Assume $H_{\CC}$ is isogenous to $\SL_r(\CC)$ \textup(or $\SL_r(\CC) \times \SL_r(\CC)$\textup).
\end{case}
Since $n + m \ge 7$, the smallest dimension of a nontrivial representation of $\so(n+m,\CC)$ is $n + m$ \cite[Exer.\ 9 and 11 in \S7.1.4, pp.~340--341]{GoodmanWallach}, so $r \ge n + m$. Thus, $n+m+\ell < 2(n+m) \le 2r$.

Let $\rho$ be a nontrivial, irreducible subrepresentation of the representation of $\sl_r(\CC)$ induced by the inclusion of~$H$ in~$G'$. (Note that $\dim \rho < 2r$.) Since $r \ge n + m > 5$, we have $\binom{r}{k} \ge \binom{r}{2} > 2r$ for $1 < k < r$. Therefore, the highest weight of~$\rho$ must be a multiple of the highest weight of either the $r$-dimensional standard representation or its dual (cf.\ \cite[Exer.~8 in \S7.1.4, p.~340]{GoodmanWallach}). This implies that $\rho$ is not self-dual. On the other hand, the representation of~$H$ in $G' = \SO(n + m+\ell,\CC)$ obviously has an invariant, nondegenerate bilinear form, and is therefore self-dual. So $\rho$ cannot be the only nontrivial, irreducible subrepresentation. This implies that $n + m + \ell$ is at least twice the minimal dimension of a nontrivial, irreducible representation of $\sl_r(\CC)$. This dimension is~$r$ \cite[Exer.~8b in \S7.1.4, p.~340]{GoodmanWallach}, so we conclude that $n + m + \ell \ge 2r$, which contradicts the conclusion of the preceding paragraph.

\begin{case} \label{E6Case}
Assume $H$ is of type~${}^1\!E_{6,2}^{28}$.
\end{case}
This means the Tits index of~$H$ is

$$    
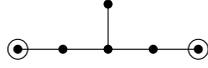
\begin{tikzpicture}[scale=0.6]
 \fill(3,1)circle(.1cm);
 \fill(1,0)circle(.1cm);
 \fill(2,0)circle(.1cm);
 \fill(3,0)circle(.1cm);
 \fill(4,0)circle(.1cm);
 \fill(5,0)circle(.1cm);
 \draw[thin](3,0)--(3,1);
 \draw[thin](1,0)--(5,0);
 \draw (1,0) circle (.22cm);
 \draw (5,0) circle (.22cm);
\end{tikzpicture}
$$

As $G \subseteq H \subseteq G'$ and $\Rrank G = n = \Rrank G'$, we know that $n = \Rrank H = 2$. Also, the anisotropic kernels of $G \doteq \SO(2,m)$ and $H \doteq {}^1\!E_{6,2}^{28}$ are (isogenous to) $\SO(m-2)$ and $\SO(8)$, respectively. As $G \subseteq H$ (and $\Rrank G = \Rrank H$), we see that $m - 2 \le 8$. Furthermore, since $H \subseteq \SO(n,m+\ell)$, we have a nontrivial representation of $H$ on~$\RR^{n + m + \ell}$.
Now note that every Weyl-orbit of nonzero weights in the $E_6$ lattice has at least 27 elements (as follows from, for example, \cite[Thm.~1.12(a)]{Humphreys}), hence every nontrivial representation of $E_6$ has dimension at least~$27$.  So
	$$ n + m + \ell \ge 27 > 8 + 2(8) \ge 8 + 2(m-2) = 2(2 + m) = 2(n + m) , $$
and consequetly $\ell \ge n + m$ as desired.
\end{proof}

\begin{rem}
It can actually be shown that Case~\ref{E6Case} of the proof of Proposition~\ref{SOinSO} cannot occur in general, regardless of what $G$ and $G'$ are. Indeed, if an almost simple Lie group of type ${}^1\!E_{6,2}^{28}$ is contained in an almost simple Lie group~$G$, then one can show that $\Rrank G \ge 3$.
\end{rem}

We now show via two examples that the assumption $n + m \ge 7$ cannot be removed from the statement of Proposition~\ref{SOinSO}.

\begin{eg} \label{SO6bad}
Let $(n,m)$ be either $(2,4)$ or $(3,3)$, and let $q = 2k+m$, for any $k \ge 1$. Then there exists
	\begin{itemize}
	\item a subgroup~$G$ of $\SO(n, q)$ that is isogenous to $\SO(n,m)$,
	and
	\item a noncocompact lattice $\Lambda$ in $\SO(n,q)$,
	\end{itemize}
such that $\Lambda \cap G$ is cocompact in~$G$.
\end{eg}

\begin{proof}
We first note that in the case of $H_{\CC}$ being isogenous to $\SL_r(\CC)$ for some~$r$, the converse of Corollary~\ref{OnlyA} is true. Namely, that $H$ is isogenous to a group with an anisotropic $\QQ$-form.
Indeed, division algebras allow the construction of anisotropic $\QQ_p$-forms of $\SL_r$, so this is straightforward.
(We also point out that more generally the converse is true and follows from \cite[Thm.~B]{BorelHarder} but we only need this special case.)

Now let $G = \SO(n,m)$. Since $n + m = 6$ and $\SO_6$ is isogenous to $\SL_4$, the above paragraph tells us that $G$ has an anisotropic $\QQ$-form~$(G)_{\QQ}$. By the classification of $\QQ$-forms of type $D_3$ (and Meyer's Theorem), we see that $(G)_{\QQ} = \SU_3(B; D, \tau)$, where $D$ is a quaternion division algebra over~$\QQ$, $\tau$ is the reversion anti-involution on~$D$, and $B$ is a $\tau$-Hermitian matrix in $\Mat_3(D)$.

Let $(G')_{\QQ} = \SU_{k+3}(B \oplus I_k; D,\tau)$.
Since $(G)_{\QQ}$ is a $\QQ$-form of $G \doteq \SL_4(\RR)$, we know that $D$ splits over~$\RR$. Therefore $\SU_k(I_k; D, \tau)$ is a $\QQ$-form of $\SO(2k)$, and so $(G')_{\QQ}$ is a $\QQ$-form of $\SO(3, 2k+3) = G'$. The 2nd vertex from the end is circled in the Tits index at every place, hence Theorem~\ref{IsotropicIffLocally} implies that $(G')_{\QQ}$ is isotropic.
\end{proof}

\begin{eg} \label{SO2rBad}
For any even $q \ge 6$, there exists
	\begin{itemize}
	\item a subgroup~$G$ of $\SO(2, q)$ that is isogenous to $\SO(2,3)$,
	\item an almost simple subgroup $H$ of $\SO(2, q)$ that contains $G$,
	and
	\item a noncocompact lattice $\Lambda$ in $\SO(2,q)$,
	\end{itemize}
such that $\Lambda \cap H$ is cocompact in~$H$.
\end{eg}

\begin{proof}
Let $H$ be the copy of $\SO(2,4)$ that is provided by Example~\ref{SO6bad}, and let $G$ be any copy of $\SO(2,3)$ in $H$.
\end{proof}

We conclude by showing that the restriction $\ell < n + m$ cannot be removed from Corollary~\ref{orthogonalcorollary}. In particular, we give a counterexample in the case when $\ell=n+m$.

\begin{eg}\label{SOequalBad}
Let $G = \SO(n,m)$, $H = G \times \SO(n + m)$, and $G' = \SO(n, n + 2m)$, so there is a natural embedding of~$H$ in~$G'$. Then there is a noncocompact lattice $\Lambda$ in~$G'$ such that $H \cap \Lambda$ is a cocompact lattice in~$H$.
\end{eg}

\begin{proof}
Let
\begin{itemize}
  \item $\sigma \bigl( a + b \sqrt{2} \bigr) = a - b \sqrt{2}$ for $a,b \in \ZZ$,
    \item $\mathcal{Z} = \bigl\{\, \bigl( \vec v, \sigma(\vec v) \bigr) \mid \vec v \in \ZZ \bigl[\sqrt{2}\bigr]^{n+m} \,\bigr\} \subseteq \RR^{2(n+m)}$,
      \item $f \colon \RR^{2(n+m)} \to \RR$ defined by
        $$f(\vec x, \vec y, \vec z, \vec w) = \sqrt{2} \sum_{i=1}^n (x_i^2 - z_i^2) - \sum_{j=1}^m (y_j^2 + w_j^2)  ,$$
  for $\vec x, \vec z \in \RR^n$ and $\vec y, \vec w \in \RR^m$,
        \item $\Lambda = \{\, g \in \SO_{2(n+m)}(f) \mid g \mathcal{Z} = \mathcal{Z} \,\}$.
          \end{itemize}

Since $f$ is a quadratic form of signature $(n, n + 2m)$, we may identify $G'$ with $\SO_{2(n+m)}(f)$. Also, since $f( \vec p ) \in \ZZ$ for all $\vec p \in \mathcal{Z}$, any linear change of basis that maps $\mathcal{Z}$ to~$\ZZ$ will turn $f$ into a quadratic form with rational coefficients, so $\Lambda$ is an arithmetic lattice in $\SO_{2(n+m)}(f)$. Moreover, as $f$ is isotropic, $\Lambda$ is noncocompact.

We may identify $H$ with the subgroup of $\SO_{2(n+m)}(f)$ that stabilizes both $\{\vec x = 0,  \vec y = 0\}$ and $\{\vec z = 0,  \vec w = 0\}$. Then
$$\textstyle H \cap \Lambda
\iso \SO_{n+m} \left( \sqrt{2} \sum_{i=1}^m x_i^2 - \sum_{j=1}^n y_j^2 ; \ZZ[\sqrt{2}] \right)$$
is the usual example of a cocompact lattice in $\SO(n,m) \times \SO(n + m)$ that is obtained by restriction of scalars.
\end{proof}

\end{appendix}



\begin{thebibliography}{xx}
\bibitem{BorelHarder}
A. Borel and G. Harder,
{\it Existence of discrete cocompact subgroups of reductive groups over local fields,}
J. Reine Angew. Math. 298 (1978), 53--64.


\bibitem{BorelTits}
A. Borel and J. Tits,
{\it El\'{e}ments unipotents et sous-groupes paraboliques de groupes réductifs. I.,}
Invent. Math. 12 (1971), 95--104.


\bibitem{Dru}{C. Drutu} \textit{Quasi-isometric classification of non-uniform lattices in semisimple groups of higher rank,} Geom. funct. anal. 19 (2000), no. 2, 327-388.
\bibitem{Esk}{A. Eskin} \textit{Quasi-isometric rigidity of nonuniform lattices in higher rank symmetric spaces,}  J. Amer. Math. Soc. 11 (1998), no. 2, 321-361.
\bibitem{FaSc} {B. Farb and R.Schwartz} \textit{ The large-scale geometry of Hilbert modular groups.} J. Differential Geom. 44 (1996), no. 3, 435-478.
\bibitem{Farb}  {B.Farb} \textit{The quasi-isometry classification of lattices in semisimple Lie groups.} Math. Res. Lett. 4 (1997), no. 5, 705-717.
\bibitem{FiWh}{D. Fisher and K. Whyte} \textit{Quasi-Isometric Embeddings of Symmetric Spaces,} arxiv.

\bibitem{GoodmanWallach}
 Goodman, Roe; Wallach, Nolan R. Symmetry, representations, and invariants. Springer, Dordrecht, 2009. ISBN: 978-0-387-79851-6.

\bibitem{Gromov} M.Gromov \textit{Infinite groups as geometric objects.} { Proceedings of the International Congress of Mathematicians, Vol. 1, 2} (Warsaw, 1983), 385�392, PWN, Warsaw, 1984.


\bibitem{Harder-BerichtNeuere}
G.\,Harder:
{\it Bericht \"uber neuere Resultate der Galoiskohomologie halbeinfacher Gruppen,}
Jber. Deutsch. Math.-Verein. 70 1967/1968 Heft 4, Abt. 1, 182--216.
\MR{0242838}

\bibitem{Humphreys}
J.\,Humphreys, \emph{Reflection groups and Coxeter groups}, Cambridge University Press, 1990.

\bibitem{KaLu} I.Kapovich and A.Lukyanenko \textit{Quasi-isometric co-Hopficity of non-uniform lattices in rank-one semi-simple Lie groups.}  Conform. Geom. Dyn. 16 (2012), 269-282.
\bibitem{KL1}{B. Kleiner and B. Leeb} \textit{Rigidity of invariant convex sets in symmetric spaces,}  Invent. Math. 163 (2006), no. 3, 657-676.
\bibitem{KL2}{B. Kleiner and B. Leeb} \textit{Rigidity of quasi-isometries for symmetric spaces and Euclidean buildings,} C. R. Acad. Sci. Paris. I Math. 324 (1997), no. 6, 639-643.
\bibitem{LMR} {A. Lubotzky and S.Mozes and M.S. Raghunathan} \textit{The word and Riemannian metrics on lattices of semisimple groups.} { Inst. Hautes \'{E}tudes Sci. Publ. Math.} No. 91 (2000), 5�53 (2001).
\bibitem{Sha}{N. Shah} \textit{Invariant measures and orbit closures on homogeneous spaces for actions of subgroups generated by unipotent elements,} arxiv.
\bibitem{Sch}{R. Schwartz} \textit{The quasi-isometry classification of rank one lattices,} Inst. Hautes \'{E}tudes Sci. Publ. Math. No. 82 (1995), 133-168 (1996).
\bibitem{Sch2} {R. Schwartz} \textit{Quasi-isometric rigidity and Diophantine approximation.} Acta Math. 177 (1996), no. 1, 75�112.
 \bibitem{Tits-Classification}
J.\,Tits
{\it Classification of algebraic semisimple groups,}
in A.\,Borel and G.\,D.\,Mostow, eds.:
\emph{Algebraic Groups and Discontinuous Subgroups (Boulder, Colo., 1965)},
Amer. Math. Soc., Providence, R.I., 1966,  pp.~33--62.

\bibitem{Tits74}
J.\,Tits
{\it Buildings of spherical type and finite BN-pairs, }
Lecture Notes in Mathematics, Vol. 386. Springer-Verlag, Berlin-New York, 1974. x+299 pp.

\bibitem{Tits-StronglyInner}
J.\,Tits
{\it Strongly inner anisotropic forms of simple algebraic groups,}
J. Algebra 131 (1990) 648--677.
\end{thebibliography}
\end{document}